\documentclass{amsart}
\usepackage{amssymb}
\usepackage{epsfig}
\usepackage{comment}

\newtheorem{thm}{Theorem}[section]

\newtheorem{lem}[thm]{Lemma}
\newtheorem{prop}[thm]{Proposition}
\theoremstyle{definition}
\newtheorem{defn}[thm]{Definition}
\theoremstyle{remark}
\newtheorem{rem}[thm]{Remark}

\numberwithin{equation}{section}
\newcommand{\eps}{\varepsilon}
\newcommand {\st}{{\rm s}}
\newcommand {\un}{{\rm u}}
\newcommand {\geodesic}{{\xi}}

\begin{document}

\title[Perturbations of geodesic flows]
{Perturbations of geodesic flows by\\ recurrent dynamics}%

\author[Marian Gidea and Rafael de la Llave ]{Marian Gidea$^\dag$ and Rafael de la Llave $^\ddag$}

\address{School of Mathematics, Institute for Advanced Study, Princeton, NJ 08540, USA, and Department of Mathematics, Northeastern Illinois University, Chicago, IL 60625, USA}%
\email{mgidea@neiu.edu}%
\address{School of Mathematics,
Georgia Institute of Technology,
Atlanta, GA 30332, USA}
\email{rafael.delallave@math.gatech.edu}
\thanks{$^\dag$Research of M.G. was partially supported by NSF grants: DMS
0601016 and DMS 0635607.}
\thanks{$^\ddag$Research of R.L. was partially supported by NSF grants: DMS
1162544.} %Index theory, Morse-Conley indices

\thanks{}
%\date{}%
%\dedicatory{}%
%\commby{}%
% ----------------------------------------------------------------
\begin{abstract} We consider a geodesic flow on a compact manifold endowed with a
Riemannian (or Finsler, or Lorentz) metric satisfying
some generic,  explicit conditions. We couple the geodesic flow
with  a time-dependent potential, driven by an external flow on some other compact manifold.
If the  external flow satisfies some  very general recurrence condition, and the potential satisfies some explicit conditions that are also very general,   we show that the coupled system has orbits whose energy grows  at a linear rate with respect to time. This growth rate is optimal.
We also show the existence of symbolic dynamics.

The existence of orbits whose energy
grows unboundedly in time is related to Arnold's diffusion problem.

A particular case of this phenomenon is obtained when the external dynamical system  is quasi-periodic, of rationally independent frequency vector, not necessarily Diophantine, thus extending  \cite{DLS06a}.
We also recover `Mather's acceleration theorem' for time-periodic perturbations
\cite{Mather,DLS00,GideaLlave06}.

Since the general class of perturbations that we consider in this paper are not characterized by a frequency of motion, we are not able to use KAM or Aubry-Mather or averaging theory. We devise a new mechanism to construct orbits whose energy evolves in a prescribed fashion. This mechanism intertwines the inner dynamics restricted to some normally hyperbolic manifold with the outer dynamics corresponding to two distinct choices of homoclinic excursions to that manifold. As this mechanism uses very coarse information on the dynamics, it is potentially of wider applicability.
\end{abstract}
%-----------------------------------------------------------------
\maketitle
% ----------------------------------------------------------------
\section{Introduction}\label{section:introduction}

%In this paper we describe a large class of time-dependent perturbations of the geodesic flow  yielding  trajectories along which the energy grows unboundedly, as well as chaotic trajectories. This phenomenon is related to the Arnold diffusion problem in Hamiltonian systems (see, e.g.,  \cite{Arnold64,ChierchiaG94,Mather,Mather2004,ChengY2004,ChengY2009,DGLS08}).

\subsection{The Hamiltonian instability problem} The phenomenon of instability in nearly integrable Hamiltonian systems was discovered in \cite{Arnold64}, which described a geometric mechanism based on whiskered tori and verified it in an example. Conceptually, this mechanism showed  how to increase the energy of a conservative mechanical system by using small perturbations. Many subsequent developments in studying the Arnold mechanism, including incorporating variational methods, can be found in the literature, e.g., in  \cite{ChierchiaG94,Bessi1996,FontichM01}.
Due to its importance to practical applications, the instability phenomenon has  also been studied by physicists using heuristic and numerical methods, e.g., in \cite{Chirikov1985,Zaslavsky,Lieberman,Tennyson}.

A remarkable revival of interest in the problem was stimulated by the unpublished work \cite{Mather}, which led to several new approaches -- variational, geometric, or mixed --  e.g., in \cite{Mather,Mather2004,Mather2012,ChengY2004,ChengY2009,Cheng2012,DGLS08,KaloshinBZ11,KaloshinZ12a,KaloshinZ12b}.

The geometric approach of \cite{DLS00,GideaLlave06,GideaR07} does  not only use whiskered tori but also normally hyperbolic invariant manifolds (NHIMs) whose stable and unstable manifolds intersect. (The method of \cite{GideaLlave06} bypasses whiskered tori completely.) The mechanism is based on interleaving homoclinic excursions (conveniently quantified by the scattering map) with the dynamics on the NHIM.

In the mechanisms discussed in the above papers \cite{DLS00,GideaLlave06,GideaR07}, there are two time scales. The homoclinic excursions are fast while the dynamics in the NHIM has slow components. If one assumes that the perturbations are periodic (or quasi-periodic) the inner dynamics can be controlled using the classical averaging method.

The main idea of this paper is that we can show Hamiltonian instability by relying almost exclusively on homoclinic excursions, through  carefully choosing the times when to jump onto the homoclinics. To apply  this novel mechanism, we do not need to assume much on the form of the perturbations.  As it turns out, by jumping onto the homoclinics at about the same place, we can obtain effects similar to averaging.  The gain of energy is obtained by timing the place of the jumps and by taking advantage of the fact that there are several homoclinics available for jumping. To concatenate these excursions, we take advantage of the method of `correctly aligned windows' which is more flexible than other shadowing techniques, and, at the same time, is precise enough to keep track of time. Two of the features of the method of correctly aligned windows that are crucial for us are that it works in systems with only partial hyperbolicity, and that it allows to concatenate segments  of trajectories.

To illustrate the approach (we hope that there are several other applications forthcoming) we study the model considered in \cite{Mather}, consisting of a geodesic flow perturbed by a time-dependent potential. Since we do not need to control the inner dynamics much we do not need averaging and we do not need to assume that the external flow is periodic or quasi-periodic. Our results also show instability in the models considered in \cite{DLS00,GideaLlave06,GideaR07} but remove several hypotheses and also give optimal diffusion time.  Of course, the orbits produced are very different from those in the above papers.   Note that the models considered in this paper do not easily admit a  quantitative averaging theory in the style of the classical one (good expositions of classical averaging theory are in \cite{Lochak,Arnold}. We do not know if the ergodic averaging theory \cite{Anosov} is enough to obtain geometric results.

Of course, we do not know whether the variational methods can be applied to general time dependent systems. Some of the geometric features of minimizers well known for periodic perturbations have counter-examples for quasi-periodically forced systems
\cite{Lions}. Of course, variational methods work under the assumption that the system is convex (i.e., its Hessian matrix is positive definite), but this is not required by our method. In this paper, we do not assume that the metric defining the geodesic flow is Riemannian, and the arguments apply just as well to Finsler or Lorentz metrics.

\subsection{Perturbed geodesic flow model}
The explicit model that we treat in this  paper is the following. We consider an (unperturbed) geodesic flow on
the unit tangent bundle of a compact manifold has a hyperbolic periodic orbit whose stable and unstable manifolds
intersect transversally. Via a potential, we couple the dynamics of the geodesic flow with that of an external flow
on some other compact manifold. We assume that the external flow has a non-trivial uniformly recurrent point.
We show that, if the coupling between the geodesic flow and the external flow satisfies some
non-degeneracy conditions, which we formulate explicitly, then the resulting system possesses some orbits
along which the energy grows to infinity linearly in time; such growth rate is optimal.
Also, we show that there exist  orbits whose energy follows a prescribed energy path.

We point out that the explicit assumptions on the geodesic
flow, on the external dynamics, and on the coupling, are very general.
They hold generically and can be checked in concrete examples
by a finite precision calculation. Also, the construction of orbits whose energy grows to infinity at an optimal rate, or of orbits whose energy behaves in some prescribed way,  is explicit, a fact which
can be potentially exploited in  concrete applications, such as in celestial mechanics.

A particular class of systems that we refer to in this paper
are quasi-periodic perturbations of the geodesic flow, which have been considered in the literature  \cite{DLS06a}.
Nevertheless, the mechanisms we use in this paper are
very different from those of the previous papers, and the
produced orbits are  different as well. This is relevant for applications, when one is not only interested in establishing the
presence  of the  `diffusion'  phenomenon, but also  in
exploring various diffusion mechanisms which can feature different characteristics \cite{Tennyson}.

An intuition to keep in mind is that the unperturbed geodesic flow describes the dynamics of a
mechanical system with a Maupertuis metric. Its energy is conserved.
The external flow, which is governed by its own intrinsic dynamics, exerts a time-dependent perturbation on the mechanical system. For the perturbed system, the energy conservation law does not hold in general. If the external flow moves with some frequency, it seems conceivable that by choosing trajectories of the geodesic flow that are in resonance with that frequency one can obtain unbounded energy growth for the perturbed geodesic flow.

The striking conclusion of this paper is that one can achieve unbounded energy growth even when the external flow has no  frequency of motion. The only assumption we use is recurrence. We can paraphrase this situation:

\begin{center}\textit {A little recurrence goes a long way.}\end{center}

\subsection{Contents of the paper}
The set up and main results appear in Section 2.
The proofs of the main results are based on a new mechanism combining geometric
methods, perturbation theory, and topological techniques.
Since all the constructions are rather explicit, we anticipate that this can be used in concrete applications, such as in celestial mechanics.

Section 3 describes the geometric method that is used in the paper.  The perturbed
system is first expressed as a slow and small perturbation of an
integrable Hamiltonian, through some rescaling of the coordinates and of the
time. Due to the assumptions on the geodesic flow, the unperturbed system possesses
a normally hyperbolic invariant manifold whose stable and unstable manifolds intersect transversally, at various places.
The transverse homoclinic intersections are used to define different  `scattering maps'.  A scattering  map is defined on a subset of  the   normally hyperbolic invariant manifold, and encodes information on
the  homoclinic trajectories.
For all sufficiently small perturbations -- which in our
case amounts to all sufficiently large energies --
the normally hyperbolic invariant manifold from the unperturbed system survives
to the perturbed system, and its stable and stable manifolds
keep intersecting transversally.  This allows us to continue the scattering maps from the unperturbed problem to the perturbed one.
The effect of a scattering map on the energy of the system
can be computed very explicitly  using the fact that the energy
is a slow variable.

In Section 4 we show that, by interspersing the `outer' dynamics, given by a scattering map, with the `inner' dynamics, given by the restriction of the flow to the normally hyperbolic invariant manifold, we can arrange that the energy changes by arbitrarily large quantities, and, in particular,  grows to infinity.

At the first stage, we construct some elementary building blocks for the two-map dynamics, each block consisting of one applications of a scattering map followed by a segment of a trajectory of the inner dynamics.
We can compute explicitly the change of energy along such an elementary building block.

At the second stage, we show that, under appropriate conditions, we can construct sequences of elementary building blocks whose energy experiences the desired changes. We also show that these effects do not happen with one scattering map, but that is crucial to use at least two scattering maps. In particular, we show that we can arrange the choices of the scattering maps and of the corresponding blocks so that  we can consistently grow the energy of the system, at an optimal rate.

These sequences of blocks determine pseudo-orbits, which are concatenation of orbits of the inner dynamics and of orbits of the outer  dynamics. They do not immediately yield true  orbits for the perturbed geodesic flow.

To prove the existence of true orbits, in Section 5  we apply the topological method of correctly aligned windows.
Around the  pseudo-orbits we construct windows that are correctly aligned, i.e.,
sequences of multi-dimensional rectangles that successively cross  one another in a non-trivial way under the inner or
the outer dynamics. A topological version of the Shadowing Lemma implies the existence of
true orbits near those pseudo-orbits.
We note that the main difficulty to carry the shadowing
argument is that the underlying dynamical system is not hyperbolic, as there
are neutral directions along the normally hyperbolic invariant manifold. Hence
the customary hyperbolic shadowing methods do not apply.

The orbits constructed in our paper are very different from the orbits previously considered in the literature.
The orbits we consider stay very little time near the unperturbed periodic orbits between successive homoclinic excursions.
This has several technical consequences: we do not need to rely on the KAM theorem, on Aubry-Mather theory, or on averaging, in order to control the behavior of the inner dynamics (the dynamics near the closed orbits between homoclinic excursions). Avoiding the use of averaging theory allows us to consider very general perturbations.  Also,   we do not need to study the geometric details on how the inner dynamics is affected by the perturbation. On the other hand,  we do have to rely more heavily on the scattering map, and to develop some sophisticated shadowing mechanisms that are not based on  hyperbolicity but rather on topological tools.

% Similar difficulties were encountered in the original problem
% of Arnold diffusion \cite{Arnold64},  when one has to shadow orbits that pass
% close to whiskered tori, as well as in several other papers. In contrast, our mechanism relies mostly on
% the scattering map and on the rearrangement of some short  orbits of the inner
% dynamics. Remarkably, we do not appeal to the obstruction
% mechanism of transition tori or of some other type of invariant sets (e.g., Aubry-Mather sets).
% For this reason we need very few hypotheses on the dynamics
% of the perturbing flow and we do not need to study in detail how
% the inner dynamics is affected by the external dynamics.

\subsection{Some other related works}
Results related  to the ones considered in this paper were proved in
\cite{Piftankin2006,PiftankinT2007}, using the method of the
separatrix map and the idea of the anti-integrable limit. See also
\cite{BolotinTreschev1999}.  The existence of orbits whose energy grows
in time at different rates, depending on the type of perturbation, was
proved in \cite{DGLS08,GelfreichT2008,Treschev04} for general Hamiltonian systems
that slowly depend on time. Their method is quite different: under
some assumptions, they identify in the frozen system two families of
periodic orbits that give rise, in the time-dependent system, to
orbits that jump from near one family to near the other in a manner
that yields unbounded growth of energy. A general discussion on the
speed of diffusion in a priori unstable Hamiltonian systems appears in
\cite{Fejoz2011}; in particular, this paper considers diffusion in the
planar elliptic restricted three-body problem. The diffusion
phenomenon in the case of the spatial circular restricted three-body
problem is considered in \cite{DGR11a,DGR11b}.

\section{Set-up and main results}
\label{section:main}
\subsection{The geodesic flow}

In this subsection we review some well known facts on the geodesic flow, and we set up the notation.

Let $M$  be an $n$-dimensional $C^r$-smooth compact manifold and $g$
be a $C^r$ Riemannian metric on $M$, where $r\geq r_0$.
(Some non-optimal values of $r_0$ are discussed in
Subsection \ref{subsection:regularity} at the end of the proofs of the main results.)   The
geodesic flow  $\geodesic: TM\times \mathbb{R}\to TM$ is the flow on the
tangent bundle $TM$ (of dimension $2n$), defined by
\[\geodesic_t(x,v)=(\xi_{x,v}(t),d\xi_{x,v}/dt(t)),\] where
$\xi_{x,v}$ is the unique geodesic with $\xi_{x,v}(0)=x$ and
$d\xi_{x,v}/dt(0)=v$. To a geodesic curve $\xi$
 on $M$    corresponds a     trajectory of the
geodesic flow given by
\[t\mapsto (\xi(t),(d\xi/dt)(t)).\]
In this paper we do not distinguish  between   a geodesic curve and   the corresponding trajectory of the geodesic flow.
Also, we do not distinguish   between  a parametrized curve and   its image.

Since the
speed along a geodesic is constant, every geodesic can be reparametrized  as a unit speed geodesic. This way we reduce the study of the geodesic flow to its  restriction  to the
unit tangent bundle $T^1M$ (of dimension $2n-1$).

Using the standard identification between  $TM$ and $T^*M$ given by $g$,  we can interpret
the geodesic flow as the Hamiltonian flow for the Hamiltonian  $H_0:T^*M\to \mathbb{R}$ given by
\begin{equation}\label{eqn:unpertham}
H_0(x,y)=\frac{1}{2}g_x(y,y).
\end{equation}
 On $T^*M$ we consider the standard symplectic
form $\omega=dy\wedge dx$.

Since $H_0$ is independent of time, it is a
conserved quantity. The energy surfaces $\Sigma_E = \{ H_0 = E\}$
are invariant under $\geodesic_t$.
We will denote by $\geodesic_{E,t}$ the
flow restricted to the energy surface $\Sigma_E$.
The restriction   $\geodesic_{1/2,t}$ of $\xi_t$ to
the energy manifold $\Sigma_{1/2}=\{H_0=1/2\}$ corresponds to the geodesic flow on
$T^1M$.

It is clear that the flow restrictions to different energy surfaces are equivalent to one another and can be obtained
just by a rescaling of time and momentum.
The mapping $D_{\sqrt{2E}}(x,y) = (x, \sqrt{2E} y ) $ -- dilation
along the fibers of $T^*M$, which is a well defined operation, because the fibers are linear spaces --
gives a diffeomorphism between $\Sigma_{1/2}$ and $\Sigma_E$.
Furthermore
\begin{equation}\label{eq:scale1}
\begin{split}
 \geodesic_{E, t}\circ D_{\sqrt{2E}} &= D_{\sqrt{2E}}\circ\geodesic_{1/2,\sqrt{2E}t},\\
 \geodesic_{ t}\circ D_{\sqrt{2E}} &= D_{\sqrt{2E}}\circ\geodesic_{\sqrt{2E}t}.
\end{split}
\end{equation}
Explicitly, the  trajectory of the geodesic flow  $\xi_E=(\xi^x_{E},\xi^y_{E})$ at energy $E$ is
related to the trajectory of the geodesic flow $\xi_{1/2}=(\xi^x_{1/2},\xi^y_{1/2})$ at energy $1/2$ by the following formula:
\begin{equation}\label{eq:scale2}(\xi^x_{E}(t),\xi^y_{E}(t))=(\xi^x_{1/2}(\sqrt{2E}\cdot t),
\sqrt{2E}\cdot \xi^y_{1/2}(\sqrt{2E}\cdot t)).\end{equation}

Below we will regard the geodesic flow as an unperturbed dynamical
system, to which we will apply an external perturbation.

\subsection{Assumptions on the geodesic flow}
We make the following assumptions on the geodesic flow.

\textsl{A1. We assume that there exists a closed, unit speed
geodesic $\lambda$, which is a hyperbolic periodic orbit for the
geodesic flow on $T^1M$.}

\textsl{A2. We assume that there exists a unit speed geodesic
$\gamma$  which is a transverse homoclinic orbit to
$\lambda$ for the geodesic flow on $T^1M$.}

Note that condition \textsl{A1} means that $\lambda$ has stable
and unstable manifolds $W^\st(\lambda)$ and $W^\un(\lambda)$, of
dimension $n$, in $T^1M$. The fact that a
geodesic $\xi$ is unit speed means that $g\left( d\xi/dt , d\xi/dt \right)=1$.
Without loss of generality, we can assume that the period of the geodesic
$\lambda$  from assumption \textsl{A1} is $1$.

The transversality condition \textsl{A2} means that
\[T_{z}W^s(\lambda)+T_{z}W^u(\lambda)=T_{z}(T^1M),\,
\textrm{ for all } z\in \gamma.\]

We note that there are many manifolds
for which the conditions \textsl{A1} and \textsl{A2} are satisfied for an
abundance of  Riemannian metrics; a partial review of  the existing results  is provided in \cite{DLS06a}. As an example,  any
surface of genus $2$ or higher, with any $C^{2+\delta}$ metric,
$\delta> 0$, has hyperbolic geodesics with transverse homoclinic
connections \cite{Katok82}. A very general result was  obtained in \cite{Contreras2007}, showing
that on any closed manifold $M$ with $\dim(M)\geq 2$ the set of
$C^\infty$ Riemannian metrics whose geodesic flow contains a
non-trivial hyperbolic basic set is $C^2$-open and $C^\infty$-dense. This  implies that the hypothesis \textsl{A1} and \textsl{A2} from above hold
for a  $C^2$-open and $C^\infty$-dense  set of Riemannian metrics.

Note that the energy $H_0$ along both the hyperbolic orbit $ \lambda$ in
\textsl{A1} and the homoclinic orbit $\gamma$ in \textsl{A2} equals~$1/2$,
but similar orbits exist in all energy surfaces because of
\eqref{eq:scale2}.

\subsection{Coupling the geodesic flow with an external dynamical system}
Next we will describe a class of perturbations of the geodesic
flow for which we will show the existence of orbits with unbounded growth of energy over time, as well as of symbolic dynamics.

We first describe an external dynamical system which we couple with the geodesic flow through a time dependent potential.
Let $X:N\to TN$ be a $C^1$-smooth
vector field on a compact manifold $N$. Let $\chi_t$ be the flow  on $N$ associated to
$X$. Let $\theta_0\in N$  and let $\chi_t(\theta_0)$ be an integral curve to $X$ with the initial condition $\theta_0\in N$.

Let \[\mathcal{V}=\{V:M\times N\to\mathbb{R}\,|\,  V \textrm{ is a } C^r-\textrm{ differentiable in $q$  and $C^1$-differentiable in $\theta$ } \}.\] We think of $V\in \mathcal{V}$ as a potential depending on the parameter $\theta$ evolving in $N$.
For every $V\in\mathcal{V}$ we  consider  a  parameter-dependent,  time-dependent Hamiltonian
$H_{\theta_0}:T^*M\times \mathbb{R}
\to\mathbb{R}$ given by
\begin{equation}\label{eqn:pertham1}
H_{\theta_0}(x,y,t)=H_0(x,y)+V(x,\chi_t(\theta_0)).
\end{equation}

\subsection{Assumptions on the external dynamical system}
We now describe some conditions on the external dynamical system.

Given $(N,\chi)$, a point $\theta_0\in N$ is said to be uniformly recurrent (or syndetically recurrent) if for every open neighborhood $U$ of $\theta_0$ there exists $T>0$ such that,  every interval $(a,b)$ with $b-a>T$ contains a time $t$ with $\chi_t(\theta_0)\in U$.  That is, a uniformly recurrent point is one which is recurrent with `bounded return times'.

The flow $\chi$  on $N$ is said to be minimal if every orbit  is dense in $N$.

Since $N$ is compact, if the flow in minimal then
for every open set $O\subseteq N$ there exists $T\geq 0$, depending on $O$, such that $\bigcup_{t\in[0,T]}\chi_t(O)=N$.

Since $N$ is compact, then there always
exists a point $\theta_0\in N$ that is uniformly recurrent for $(N, \chi)$.  This follows from the Poincar\'e Recurrence Theorem.

Moreover, if $(N,\chi)$ is minimal then every point $\theta\in N$ is uniformly recurrent.

Conversely,  if $\theta_0\in N$ is uniformly recurrent then its orbit closure   is a minimal set.

If every point $\theta\in N$ is uniformly recurrent then $N$ is the disjoint union of its minimal subsystems, in which case $N$ is called semi-simple. See \cite{Furstenberg81}.

The following two alterative conditions will be used in the statements of the main results.

\textsl{A3. We assume that the flow $\chi_t$  has a uniformly recurrent point $\theta_0\in N$ with $X(\theta_0)\neq 0$.}

\textsl{A3'. We assume that the flow $\chi_t$  is minimal  on $N$.}

\subsection{Statement of the results}
The following result  provides the existence of trajectories with unbounded growth of energy for the system \eqref{eqn:pertham1}.

\begin{thm}\label{thm:main1}
Let $g$ be a Riemannian metric on $M$ satisfying the conditions \textsl{A1},
\textsl{A2}, and let $(N,\chi)$  be an external dynamical system satisfying
\textsl{A3}. Then there exist $\theta_0\in N$, and   $\mathcal{V}'$,  $C^2$-open and $C^r$-dense in  $\mathcal{V}$ with respect to the $C^r$-topology, $r \ge r_0$, such that, for every $V\in\mathcal{V'}$,   the system \eqref{eqn:pertham1} has a
solution  with $H_{\theta_0}(x(t),y(t),t)\geq At+B$, for some $A,B\in\mathbb{R}$ with $A>0$, and for all $t$ sufficiently large.
\end{thm}

For example, the condition \textsl{A3'} in Theorem \ref{thm:main1} is automatically satisfied if  $X(\theta)\neq 0$ for all $\theta$. Note that in this case the  only requirement  on the flow $\chi$ on $N$ is that it does not have any fixed points. This condition is necessary since we want,  as we shall see in the argument in Subsection \ref{subsection:gainseq}, that the flow line $\chi_t(\theta_0)$ leaves some neighborhood of  the point $\theta_0$ before it returns again to that neighborhood.

\begin{thm}\label{thm:main2}
Let $g$ be a Riemannian metric on $M$ satisfying the conditions \textsl{A1},
\textsl{A2}, and let $(N,\chi)$  be an external dynamical system satisfying
\textsl{A3'}.
Then there exists  a set $\mathcal{V}'$,  open and dense in  $\mathcal{V}$ with respect to the $C^r$-topology, $r \ge r_0$, such that, for every $V\in\mathcal{V'}$,  and every $\theta_0\in N$, the system \eqref{eqn:pertham1} has a solution  for which the energy $H_{\theta_0}(x(t),y(t),t)$ grows
linearly to infinity as $t$ tends to infinity, i.e, $H_{\theta_0}(x(t),y(t),t)\geq At+B$ for some $A,B\in\mathbb{R}$ with $A>0$, and for all $t$ sufficiently large.
\end{thm}

The hypothesis \textsl{A3} in  Theorem \ref{thm:main1} requires one to choose the parameter $\theta_0$  to be a non-trivial uniformly recurrent point, and leads an unstable trajectory corresponding to that particular choice, while the hypothesis \textsl{A3'} in Theorem \ref{thm:main2} allows the parameter $\theta_0$  to be arbitrary.

We also note that Theorem \ref{thm:main2} remains valid under the weaker assumption that  the flow $\chi_t$ on $N$ is  semi-simple.

The linear growth rate $H_{\theta_0}(x(t),y(t),t)\approx
t$ in Theorem \ref{thm:main1} and Theorem \ref{thm:main2} is optimal.
Indeed, the energy $H_{\theta_0}(x(t),y(t),t)$
cannot grow in time faster than linearly, as we can easily show.  By \eqref{eqn:pertham1}, we have that
\[\frac{d}{dt} H_{\theta_0}(x(t),y(t),t)= \frac{\partial V}{\partial
  x} (x(t),\chi_t(\theta_0))X(\chi_t(\theta_0)),\] which is bounded due
to the compactness of $M$ and $N$.

In Subsection \ref{subsec:buildingblock}, we  provide   an explicit condition  \textit{A4} that ensures  $V\in \mathcal{V}'$, as in  Theorem    \ref{thm:main1} and  Theorem   \ref{thm:main2}.  We emphasize here that this condition  amounts to an explicit computation \eqref{eqn:A4} that is verifiable in concrete systems. The potentials $V$ satisfying this condition  form  a $C^2$-open and $C^\infty$-dense set in  $\mathcal{V}$.
The condition  \textit{A4} depends on a hyperbolic closed geodesic and a on a pair of geometrically distinct homoclinic  orbits associated to it. A generic geodesic flow has infinitely many homoclinic orbits to the same hyperbolic closed geodesic.
Of course, for the main results we only need to verify the condition  for just  a single  pair of homoclinic orbits.

Besides orbits whose energy grows unboundedly in time, there  also exist
orbits whose energy makes chaotic excursions, i.e., they follow any prescribed energy path. In other words,  the system admits symbolic dynamics.

\begin{thm}\label{thm:symbolic}
Assume the conditions of Theorem~\ref{thm:main1}, and a potential $V\in\mathcal{V'}$, with the set $\mathcal{V'}$
as in the statement of that theorem.
Let $E_*$ be sufficiently large.  There exists  $C> 0 $, depending on the potential $V$,
such that, for any function $\mathcal{E}: [0,\infty) \rightarrow [E_*, \infty)$,  such
that $|\mathcal{E}'| \le C$, there exist $\theta_0\in N$, a time reparametrization $\mathcal{T}:[0,\infty)\to[0,\infty)$,  and a solution  $(x(t),y(t))$  of the system  \eqref{eqn:pertham1}, and a constant $D>0$ such that
\[
|H_{\theta_0}(x(\mathcal{T}(t)),y(\mathcal{T}(t)))  - \mathcal{E}(t) | \le D \mathcal{E}(t)^{-1/2}.
\]\end{thm}

\subsection{Some examples of applications.}

The conditions for the flow on $N$ in Theorem \ref{thm:main1} are very general.
What makes the examples presented below  more surprising is that
the flow on $N$ can have a
very simple orbit structure. The results include, as a particular case,
quasi-periodic forcing, in which case we have the presence of the KAM and Nekhorosev phenomena,
which prevent, or delay, the onset of linear growth of energy, for a positive measure set of orbits. This is why we single  out some examples
to showcase Theorem \ref{thm:main2}.

%Theorem ~\ref{thm:main1} also applies to other systems with a more complicated orbit structure (e.g., transitive Anosov systems).
For systems with a rich orbit structure (e.g., horseshoes giving
rise to   symbolic dynamics), there are
simpler arguments that show that one can get instability. Roughly,
if the forcing system has essentially arbitrary orbits, we can
choose initial conditions in the forcing that lead to orbits that
always increase the energy. Hence, for systems with complicated orbit
structure there are many results of instability. A representative
paper of this line of reasoning  is \cite{Ortega2006}. Of course, even in the case when simpler mechanisms apply, the orbits constructed here are different, since we rely on the homoclinic excursions and not on riding the external forcing.

\subsubsection{Perturbation of the geodesic flow by a quasi-periodic potential}\label{subsection:quasiperiodic}

We consider the particular case when
$N=\mathbb{T}^d=\mathbb{R}^d/\mathbb{Z}^d$, and the flow on $N$  is a linear flow of rationally independent  frequency vector $\nu$, i.e.,
$\nu\cdot k\neq 0$ for all $k\in \mathbb{Z}^d$. Such a flow can be
written as $\chi_t(\theta_0)=\theta_0+\nu\cdot t$, for any initial
point $\theta_0\in \mathbb{T}^d$; for simplicity, we let
$\theta_0=0$.  The flow $\chi_t$ is
minimal on $N=\mathbb{T}^d$.

The corresponding  perturbed dynamical system  is described by the
time-dependent Hamiltonian $H:T^*M\times \mathbb{R}\to\mathbb{R}$
given by
\begin{equation}\label{eqn:pertham2}H(x,y,t)=H_0(x,y)+
V(x,\nu t).\end{equation}

As a consequence of Theorem \ref{thm:main1}, we obtain that, for a generic set of $C^r$-smooth potentials
$V:M\times \mathbb{T}^d  \to\mathbb{R}$, the system \eqref{eqn:pertham2} has solutions $(x(t),y(t))$ for which the energy
$H(x(t),y(t),t)$ grows linearly to infinity as $t$ tends to
infinity. Note that the frequency $\nu$ is not required to be Diophantine, as in \cite{DLS06a}. In addition to recovering the results from \cite{DLS06a} under   weaker conditions, we also obtain the existence of orbits whose energy grows at an optimal speed.
Again, we point out that the orbits constructed in this paper are very different from those in \cite{DLS06a}, where the orbits stay long times near the closed geodesics.

When $d = 1$, the flow $\chi_t$ describes a periodic motion on $\mathbb{T}^1$,  thus we obtain, as a particular case, the geodesic flow perturbed by a generic, periodic potential considered in, e.g., \cite{Mather,BolotinTreschev1999,DLS00,Kaloshin}.

\subsubsection{Flows on Lie groups.}
In this section, we describe  some other examples of external dynamical systems $(N,\chi_t)$ that  can be used in Theorem \ref{thm:main2}. These examples are mild enough so that the mechanisms in \cite{Ortega2006} do not apply but nevertheless there is no averaging theory for them.

Note that we can interpret the example in Subsection \ref{subsection:quasiperiodic} as a flow on a Lie group generated  by a left-invariant  vector field. Given a compact Lie group $N$,  recall that a left translation on $N$ is a map $L_g:N\to N$ given by $L_g(\theta)=g\theta$ for some $g\in N$. A vector field $X$ on $N$ is called left-invariant if $X$ is invariant with respect to all left-translations,  i.e.
$(L_g)_{*}(X)=X$ for all $g\in N$.  Let $\chi_t$ be the flow of $X$ on $N$. It is well known that each integral curve of $\chi_t$ is homeomorphic to the integral curve through the identity, that is,  $\chi_t(\theta)=\theta\cdot\chi_t(e)$.  The flow  $\chi_t$ is minimal if and only if $N$ is Abelian. However, every compact Abelian Lie group is a torus.  Thus, our example in  Subsection \ref{subsection:quasiperiodic} has a natural interpretation as a left-invariant flow on a compact, Abelian,  Lie group.

\subsubsection{Horocycle flows.}
Let $P$ be a compact connected manifold with a Riemannian metric of negative curvature, and  $\xi_t$ be the geodesic flow restricted to the unit tangent bundle $T^1P$. The horocyle flow is the unit speed flow on $T^1P$ whose orbits (referred as horocycles) are the strong stable manifolds $\{W^{ss}(z)\}$ of the geodesic flow
\[W^{ss}(z)=\{z'\in T^1M\,|\, \lim_{t\to\infty}d(\xi_t(z'),\xi_t(z))=0\}.\]

As a particular case, assume that $P$ is a surface of constant negative curvature. The universal covering space is the Poincar\'e upper half-plane
$\mathbb{H}$ with the  Poincar\'e  metric denoted $ds$. The geodesics in $\mathbb{H}$ are vertical lines and circles
orthogonal to the real axis, and  the horocycles are horizontal lines and circles
tangent to the boundary. The orientation-preserving isometries of $\mathbb{H}$ are the linear fractional transformations, i.e., the elements of \[PSL(2,\mathbb{R})=\{z\mapsto\frac{az+b}{cz+d}\,|\,ad-bc=1 \}.\]
We can identify $P$ as $PSL (2, \mathbb{R})/\Gamma$ where $\Gamma$ is a discrete co-compact subgroup\footnote{$\Gamma$ is co-compact if $PSL (2, \mathbb{R})/\Gamma$ is compact.}
of $PSL (2, \mathbb{R})$.    The geodesic flow is given by
\[\xi_t(z)=\left(
               \begin{array}{cc}
                 e^t & 0 \\
                 0 & e^{-t} \\
               \end{array}
             \right)z\Gamma
\]
and the horocyle flow
\[\chi_t(z)=\left(
               \begin{array}{cc}
                 1 & t \\
                 0 &1 \\
               \end{array}
             \right)z\Gamma.\]
When the manifold $P$ has  (variable)
negative curvature, the universal Riemannian covering surface is
the upper half-plane with metric $f^2ds$ where $f$ is a non-zero
function. Then $P$ can be regarded as the quotient
of this space via a discrete co-compact group of isometries.

The geodesic flow on the unit tangent bundle is Anosov. Hence the unit tangent bundle is foliated by foliated by stable and unstable manifolds. When $P$ is a surface, the stable and unstable manifolds are $1$-dimensional. The horocycle flow is the motion at unit speed along the stable/unstable manifolds.

Hedlund \cite{Hedlund} proved that the horocycle flow on a compact surface of constant negative curvature  is minimal. Furstenberg  \cite{Furstenberg73} showed that it is uniquely ergodic (i.e., it admits  a unique ergodic measure). B. Marcus proved the same result for  compact surfaces of variable negative curvature \cite{Marcus}. Results on the minimality of the horocycle flow in higher dimensions were obtained by Eberlein, see e.g. \cite{Eberlain}. We note that for surfaces the horocycle foliation is $C^{2-\eps}$ \cite{GerberW99} but in general not $C^2$. This is enough for our result as argued in Subsection \ref{subsection:regularity}.

Hence, an interesting class of examples of   external dynamical systems $(N,\chi)$ that can be used in Theorem \ref{thm:main1} comes from horocycle flows on  compact connected manifolds of negative curvature. We remark that  horocycle flows do not determine perturbations of a fixed frequency (as in the quasi-periodic case).

\subsubsection{Flows on homogeneous spaces.}
Let $G$ be a Lie group, $\Gamma$ a discrete subgroup, $g:\mathbb{R}\to
\Gamma$ a one-parameter subgroup, and $\chi:\mathbb{R}\times G/\Gamma
\to G/\Gamma$ be the $G$-induced flow, given by
$\chi_t(z\Gamma)=(g(t)z)\Gamma$.  $G/\Gamma$ with the $G$-induced flow
is a homogeneous space.  An element $g\in G$ is unipotent if the adjoint\footnote{The adjoint $\textrm{Ad}(g)$ of $g$ is the
  derivative of the map $\Psi_g:G\to G$, $z\to gzg^{-1}$ at $e$, i.e.,
  $\textrm{Ad}(g)=(d\Psi_g)_e:T_eG\to T_eG$, where $d$ is the
  differential and $T_eG$ is the tangent space at the origin $e$.}
$\textrm{Ad}(g)$ of $g$ is
a unipotent matrix, i.e., $1$ is its only eigenvalue. If $\chi_t$ is
unipotent for each $t$, then the flow $\chi$ is called a unipotent
flow.

Ratner classification theory \cite{Ratner91} asserts that, in the case
when $\Gamma$ is co-compact, and the flow $\chi$ is unipotent, if
$\chi$ is uniquely ergodic then it is minimal. If
$\textrm{Vol}(G/\Gamma)<\infty$, where $\textrm{Vol}$ denotes the Haar
measure induced by $\chi$, then the properties of minimality and
unique ergodicity are equivalent.

Hence, another interesting class of examples of   external dynamical
systems $(N,\chi)$ that can be used in Theorem \ref{thm:main1} are
unipotent flows on  homogeneous spaces of finite volume which are
uniquely ergodic. Some specific examples appear in, e.g.,
\cite{Starkov95,Mozes98,Weiss01}.

\subsubsection{Celestial mechanics and astrodynamics.}
We consider the Kepler problem, described by the Hamiltonian $H(y,x)=\frac{1}{2}|y|^2-\frac{1}{|x|}$.
The solutions are  conic sections or collision orbits. The regularized Keplerian motions for $H<0$ can be lifted to  trajectories of the geodesic flow on $S^2$.  Let us consider this problem as a model for the motion of a satellite around the Earth.
Following \cite{ConterasP02}, the set of $C^\infty$ Riemannian metrics on $S^2$
whose geodesic flow contains a non-trivial hyperbolic basic set is open and dense
in the $C^2$ topology. Thus, we can choose a  Riemannian metric that approximates the standard $S^2$-metric and for which there exists a hyperbolic closed geodesic with transverse homoclinic connection, satisfying conditions \textsl{A1}, \textsl{A2}. A  closed geodesic will correspond to a closed orbit around the Earth. The motion of Moon and Sun can be regarded as a  quasi-periodic forcing, as in \textsl{A3}. It seems possible that the method used to prove Theorem \ref{thm:main2}, which is constructive, can be adapted to this example in order to  design explicit maneuvers on how to  move, in specific ways, the satellite from one closed orbit to another around the Earth, or to  increase the size and shape of a satellite  orbit.
The fact that an approximation to the metric on $S^2$ is used in this argument, will result in some  errors  that can be corrected by low energy maneuvers.
Satellite trajectory repositioning is very useful in astrodynamics, see, e.g., \cite{Ocampo05}. A paper of related interest  is \cite{Albers12}.

\section{Geometric method}\label{section:geometric}

In this section we review briefly the geometric method to study
perturbations of the geodesic flow developed in \cite{DLS00}.
This method is based on the theory of normally hyperbolic
invariant manifolds and on the scattering map.
We will take the general setup from  \cite{DLS00} but we will make substantial modifications.

We consider the parameter-dependent Hamiltonian
$H:T^*M\times N\to\mathbb{R}$ given by  $H(x,y,\theta)=H_0(x,y)+V(x,\theta)$.
The Hamilton  equations and the parameter evolution equation  are
\begin{equation}\label{eqn:parham}
\begin{split}
\frac{dx}{dt}&=\frac{\partial H_0}{\partial y},\\
\frac{dy}{dt}&=-\frac{\partial  H_0}{\partial  x}-\frac{\partial V}{\partial  x},\\
\frac{d\theta}{dt}&=X(\theta).
\end{split}\end{equation}

We write the corresponding flow  $\psi:T^*M\times N\times\mathbb{R} \to T^*M\times N$ as \begin{equation}\label{eqn:extfl}
\psi_t(x,y,\theta)=(\geodesic_t(x,y,\theta), \chi_t(\theta)),
\end{equation}
where $\chi_t$ is the flow defined by the vector field $X$ on $N$.

\subsection{Normal hyperbolicity}
\label{subsection:normalhyper}
First we consider the unperturbed system, which is
described by the Hamiltonian $H_0$ given by \eqref{eqn:unpertham}.
Each energy manifold $\Sigma_E=\{(x,y)\,|\, H_0(x,y)=E\}$ is
invariant under the geodesic flow $\geodesic_t$ (now viewed  as the
Hamiltonian flow of $H_0$). We will denote by
$\xi_{E,t}=(\xi^x_{E,t} , \xi^y_{E,t} )$ a trajectory of the geodesic flow lying on
$\Sigma_E$. Due to the  rescaling property of the geodesic flow,
the assumptions \textsl{A1} and \textsl{A2} imply that for each
energy level $E$ there exists a periodic orbit $\lambda_E$ that is
hyperbolic in $\Sigma_E$, and a transverse homoclinic
orbit $\gamma_E$ to $\lambda_E$ in $\Sigma_E$.

We consider a sufficiently large initial energy level $E_*\geq 0$ (to be specified later in the argument),  and we
define the $2$-dimensional cylinder $\Lambda_0=\bigcup_{E\geq
E_*}\lambda_E$ in $T^*M$. Note that $\Lambda_0$ is a manifold with boundary but the flow is tangent to the boundary.
This is a normally hyperbolic invariant
manifold for the Hamiltonian flow on $T^*M$, whose stable and
unstable manifolds are given by $W^s(\Lambda_0)=\bigcup_{E\geq
E_*}W^s(\lambda_E)$ and $W^u(\Lambda_0)=\bigcup_{E\geq
E_*}W^u(\lambda_E)$, respectively.
The stable and unstable manifolds
of $\Lambda_0$ intersect transversally along the $2$-dimensional
homoclinic manifold $\Gamma_0=\bigcup_{E\geq E_*}\gamma_E$ in
$T^*M$.

\subsection{Scaled coordinates}
\label{subsection:scaled}
We now rescale the coordinates $(x,y)$  and
the time $t$ so that, for high energies,  the flow corresponding
to the rescaled Hamiltonian  is a small and slow perturbation of
the geodesic flow. For  $E_*$ sufficiently large
 we introduce a new parameter
$\eps=1/\sqrt{E_*}$; we note that $E_*\to\infty$ if  and only if
$\eps\to 0$.

The rescaled coordinates are $(q, p)$ defined by $ q=x$, $p=\eps y$, and the rescaled time $s$ is
given by $s=t/\eps$. The variable $\theta$ remains unchanged. The parameter-dependent Hamiltonian in these new
variables  is
\begin{equation}\label{eqn:hamrescaled}
    H_\eps(q,p, \theta)=H_0( q,p)+\eps^2V( q, \theta)=\eps^2 H(x,y,\theta).
\end{equation}

The corresponding Hamilton equations and the parameter evolution equation   are
\begin{equation}\label{eqn:hamrescaledeq}
    \begin{split}
\frac{dq}{ds}&=\frac{\partial H_0}{\partial p}\\
\frac{dp}{ds}&=-\frac{\partial  H_0}{\partial
 q}-\eps^2\frac{\partial V}
{\partial q}\\
\frac{d \theta}{ds}&=\eps X(\theta).
    \end{split}
\end{equation}
The corresponding flow $\psi^\eps_{s}$   is of the form $\psi^\eps_{s} =(\geodesic^\eps_{s}(
q, p, \theta), \chi^\eps_{s}(\theta))$.

When $\eps\to 0$, the flow $\chi^{\eps}_s$ approaches the constant flow $\textrm{id}_\theta$ on $N$, and the flow $\geodesic^\eps_{s}$ approaches the flow $\geodesic_{s}$ in the $C^{r-1}$ topology.
The limiting flow \[\psi_{s}=(\geodesic_{s}, \textrm{id}_{\theta})\]
on the extended phase space $T^*M\times N$ has a  normally hyperbolic invariant manifold
$\tilde\Lambda_0=\bigcup_{E\geq E_*}\lambda_E\times N$, since the exponential expansion  rates of $\geodesic_{s}$ are
larger than those of  $\textrm{id}_\theta$.  The stable and unstable manifolds of $\tilde \Lambda_0$ are
given by $W^s(\tilde\Lambda_0)=\bigcup_{E\geq
E_*}W^s(\lambda_E)\times N$ and $W^u(\tilde
\Lambda_0)=\bigcup_{E\geq E_*}W^u(\lambda_E)\times N$, respectively.
Obviously, $W^u(\tilde\Lambda_0)$ and $W^s(\tilde\Lambda_0)$  intersect
transversally along the $(2+d)$-dimensional homoclinic manifold
$\tilde\Gamma_0=\bigcup_{E\geq E_*}\gamma_E\times N$ in $T^*M\times
N$.

Now we refer to  the theory of normal hyperbolicity (see \cite{HPS,Fenichel74,Bates2000})
that  shows that the invariant manifolds, which were identified in the
limiting system, survive as locally invariant manifolds for the perturbed system.  A very explicit proof of this can be found in \cite{DLS06a}. The locally invariant
manifolds are exactly invariant for a modified system constructed explicitly in \cite{DLS06a}.
The modifications are supported on $E \le E_0$.
When we refer to stable and unstable manifolds,
we mean the stable and unstable manifolds of the extended
system.
Since we will be considering orbits whose energy stays large enough
(in particular in regions where the extended system agrees with the
original system, the orbits we construct will also be orbits of the original system.

\begin{rem}
When the system is differentiable enough and the perturbation is
periodic or Diophantine,  it is easy to remark \cite{DLS00} that
the invariant manifold will contain KAM tori that act as boundaries, so that the manifold
is in fact invariant. Of course, the argument above does not require any differentiability
and works for general perturbations.
\end{rem}

Since $\eps$
enters both in the size of the perturbation parameter and in the
time reparametrization,  in order to apply the standard normally
hyperbolicity theory one can rewrite the perturbed Hamiltonian as a
two parameter problem, with one parameter for the size of the
perturbation and the other for the time change, and prove the
persistence of the normally hyperbolic, locally invariant manifold from the unperturbed
case to the perturbed case for all small enough sizes of the
perturbation and uniformly in the time change parameter (see
\cite{DLS00,DLS06a}). In the end, it follows that, for all
sufficiently small $\eps$ (and, implicitly, for all sufficiently
large $E_*$) the manifold $\tilde\Lambda_0$ can be continued to  a
normally hyperbolic locally invariant manifold
$\tilde\Lambda_\eps=\Lambda_\eps\times N$ for the flow
$\psi^{\eps}_{s}$.

Moreover, the manifold $\tilde\Lambda_\eps$ depends smoothly on the parameter $\eps$, in the sense that
there exists a $C^{r-1}$-smooth parametrization
$\tilde k_\eps:\tilde\Lambda_0\to\tilde \Lambda_\eps$ of $\tilde \Lambda_\eps$, of the type $\tilde k_\eps=(k_\eps, \textrm{id}_{\theta})$, which depends $C^{r-2}$-smoothly on the parameter, such that
 $\tilde k_\eps(\tilde\Lambda_0)=\tilde\Lambda_\eps$ (and, in particular, $k_\eps(\Lambda_0)=\Lambda_\eps$).

In addition,  there exist stable and unstable
manifolds $W^u(\tilde\Lambda_\eps)$ and $W^s(\tilde\Lambda_\eps)$
that vary $C^{r-1}$-smoothly with $\eps$. Similarly, there exist $C^{r-1}$-smooth parametrizations $\tilde k^s_\eps :W^s(\tilde\Lambda_0)\to W^s(\tilde \Lambda_\eps)$, $ \tilde k^u_\eps:W^s(\tilde\Lambda_0)\to W^s(\tilde \Lambda_\eps)$ of the local stable and unstable manifolds of $\tilde\Lambda_\eps$, respectively, which agree with $\tilde k_\eps$ on $\tilde \Lambda_0$.

Since transversality is an open
condition, then $W^u(\tilde\Lambda_\eps)$ and $W^s(\tilde
\Lambda_\eps)$ intersect transversally along a locally unique
homoclinic manifold $\tilde\Gamma_\eps=\Gamma_\eps\times N$, for all
$\eps$ sufficiently small (and so  for all sufficiently
large~$E_*$).

\subsection{Action-angle coordinates}

For the unperturbed system, on the normally hyperbolic invariant manifold $\Lambda_0$ we can
put a system of action-angle coordinates $(J,\phi)$, where
the action coordinate is $J=\sqrt{2E}$, and the angle coordinate
$\phi\in\mathbb{T}^1$ is symplectically conjugate with $J$, i.e., $dJ\wedge d\phi= (dy\wedge dx)_{\mid \Lambda_0}$.
The unperturbed Hamiltonian is integrable on $\Lambda_0$,  and it takes the following form in  the action-angle coordinates
\begin{equation}\label{eqn:unpertcoordsdham}
    H_0(\phi,J)=\frac{1}{2}J^2.
\end{equation}
The unperturbed Hamiltonian flow on $\Lambda_0$ takes the form
$J(t)=J,\phi(t)=\phi_0+Jt$. Thus $\Lambda_0$ is
foliated by invariant tori $\mathcal{T}_{E}=\{(J,\phi)\,|\, J=\sqrt{2E}, \phi\in
\mathbb{T}^1\}$ corresponding to each energy level $H_0=E\geq E_*$.

For the perturbed system, the
action-angle coordinate system $( J, \phi)$ on
$\Lambda_0$  can  be continued via $k_\eps$  to an action-angle coordinate system $(
J_\eps, \phi_\eps)$ on
$\Lambda_\eps$. In these coordinates, the perturbed Hamiltonian function restricted to $\Lambda_\eps$ takes the form
\begin{eqnarray}\label{eqn:rescaledham}
    H_\eps(\phi_\eps , J_\eps,\theta )&=&
\frac{1}{2} J_\eps ^2+\eps^2
V(\phi_\eps , J_\eps,\theta ),
\end{eqnarray}
where $V(\phi_\eps ,
J _\eps,\theta )$ is obtained by expressing $V(\phi, J, \theta)$ in these coordinates.
The Hamilton  equations and the parameter evolution equation are:
\begin{equation}\label{eqn:actionanglesystem}
    \begin{split}
\frac{d\phi_\eps}{ds}&= J_\eps+\eps^2\frac{\partial   V}{\partial J_\eps}\\
\frac{d J_\eps}{ds}&=-\eps^2\frac{\partial  V}
{\partial\phi_\eps}\\
\frac{d\theta}{ds}&=\eps X(\theta).
    \end{split}
\end{equation}

Note that, in general,  the foliation of $\Lambda_0$  by invariant tori $\mathcal{T}_{E}$ in the unperturbed case does not survive to the perturbed case.  For a perturbation driven by a general external flow $\chi$ on $N$ we cannot apply the KAM theorem since the perturbation affecting the tori is not necessarily periodic/quasiperiodic. For the same reason, we cannot apply averaging theory as in \cite{GideaLlave06}, since we cannot obtain level sets of the action that remain almost-invariant for sufficiently long time.

\subsection{The scattering map}\label{subsection:scattering-general}

A key tool to study the dynamics of a normally hyperbolic invariant
manifold with a transverse
homoclinic manifold is the scattering map \cite{DLS00}.
The role of the scattering map
is to relate the past asymptotic trajectory of an orbit in the homoclinic
manifold to its future asymptotic trajectory. An extended study of
the scattering  map and its geometric properties can be found in \cite{DLS06b}.

Here we briefly recall the construction of the scattering map. We consider a  flow on some manifold $P$. Let $\Lambda$ be a normally hyperbolic
invariant manifold for the flow,  with the
unstable and stable manifolds $W^u(\Lambda)$ and $W^s(\Lambda)$ intersecting
transversally along a  homoclinic manifold $\Gamma$. This means that
$\Gamma\subseteq W^u(\Lambda) \cap W^s(\Lambda)$ and, for each $z\in \Gamma$, we have
\begin{equation}\begin{split}\label{scattering1}
T_zP=T_zW^u(\Lambda)+T_zW^s(\Lambda),\\
T_z\Gamma=T_zW^u(\Lambda)\cap T_zW^s(\Lambda).
\end{split} \end{equation}

By the normal hyperbolicity of $\Lambda$, $W^u(\Lambda)$ is foliated by $1$-dimensional fibers $W^u(z)$, $z\in\Lambda$, and
 $W^w(\Lambda)$ is foliated by $1$-dimensional fibers $W^w(z)$, $z\in\Lambda$.
For each $z\in W^u(\Lambda)$ there exists a unique $z^-\in \Lambda$ such that
$z\in W^u(z^-)$, and for each   $z\in W^s(\Lambda)$ there exists
a unique $z^+\in \Lambda$ such that $z\in W^s(z^+)$.
We define the  wave maps  \begin{equation}\label{eqn:wavemaps}\begin{split}
\Omega^{+}:W^s(\Lambda)\to \Lambda,\,
\Omega^{+}(z)=z^{+},\\ \Omega^{-}:W^u(\Lambda)\to \Lambda,
\Omega^{-}(z)=z^{-}.\end{split}\end{equation}  These maps are differentiable.

To define the scattering map, we
make the additional assumption  that for each $z\in \Gamma$ we
have \begin{equation}\begin{split}\label{scattering2}
 T_zW^s(\Lambda)=T_zW^s(z^+)\oplus T_z(\Gamma),\\
T_zW^u(\Lambda)=T_zW^u(z^-)\oplus T_z(\Gamma).
\end{split} \end{equation}

By restricting $\Gamma$ to some open subset of it, if necessary, we can ensure that the restrictions of $\Omega^\pm$ to $\Gamma$ are diffeomorphisms. We define the scattering map associated to
$\Gamma$  to be the diffeomorphism
$S =\Omega^+\circ (\Omega^-)^{-1}$ from
$U^-:= \Omega^- (\Gamma)$ in $\Lambda$ to  $U^+:=\Omega^+(\Gamma)$ in $\Lambda$.

As pointed out in \cite{DLS06b}, the scattering map can be defined in an analogous way in the case of
time-dependent systems. In this case, one considers a flow associated to a skew-product vector field $(Y(z,\theta), X(\theta))$ on a product manifold $P\times N$, where the skew product vector field is assumed to be close  to an autonomous vector field, i.e.,
\[\|Y(z,\theta)-Y_0(z)\|_{C^r}\ll 1\]
for some vector field $Y_0(z)$ on $N$.

Assume that there  exists a normally hyperbolic invariant manifold $\Lambda_0$ in $P$ for the flow of $Y_0(z)$. Let $S_0$ be the scattering map associated to a homoclinic channel $\Gamma_0$. Assuming that the exponential rates of the flow on $N$ are smaller than the exponential rates for the flow on $P$, it follows that  $\Lambda_0\times N$ is normally hyperbolic for the product flow of $(Y_0(z),X(\theta))$, that $\Gamma_0\times N$ is a homoclinic channel, and the corresponding scattering map $\tilde S_0$ is a product of the form $\tilde{S}_0(z,\theta)=(S_0(z), \theta)$.

By the  theory of normal hyperbolicity, there  exists a normally hyperbolic invariant manifold close to $\Lambda\times N$ in $P\times N$ for the flow of  $(Y(z,\theta), X(\theta))$, provided $Y(z,\theta)$ is sufficiently close to $Y_0(z)$, in the sense described above. There also exist a  homoclinic channel close to $\Gamma\times N$. The corresponding scattering map takes the skew product form $\tilde{S}(z,\theta)=(S(z,\theta), \theta)$.

\subsection{The scattering map for the unperturbed geodesic flow}
\label{subsection:scatterignunpert}
In the case of the unperturbed geodesic flow, described by the
Hamiltonian $H_0$, the condition \eqref{scattering1} is satisfied
due to the assumption \textsl{A2}, and the condition
\eqref{scattering2} follows from the fact that $ \Lambda_0$ is
foliated by the invariant tori   ${\mathcal{T}}_E$. The
scattering map preserves each of these invariant tori and it only
changes the phase $\phi$ along each  torus by an amount $a$ that is independent of the  torus.
The scattering map $S_0:U_0^-\to U_0^+$ is expressed relative to the $( J,  \phi)$-coordinates on
$\Lambda_0$  by
\begin{equation}\label{eqn:scatteringcoords}
    S_0( J^- , \phi^-)=(J^+,\phi^+)=(  J^-,\phi^-+a),
\end{equation}
where $J^-=J^+$ and  the phase shift $a= \phi^+- \phi^-$ is independent
of the point $z^-=z^-( \phi^-, J^- )$ in $U_0^-$, as it only depends on the
homoclinic manifold  $\Gamma_0$.

Remarkably, in this case the scattering map  can be  globally defined
as a continuous map on the whole of $ \Lambda_0$, as it has no monodromy (see \cite{DLS06b}). The continuation, however, is rather subtle because
when the base point moves along a non-contractible closed curve in $\Lambda$ the point of homoclinic intersection changes.
This causes, as observed already in \cite{DLS00} that after
perturbations, this global definition may be impossible.
When we continue the scattering map along a non-contractible closed curve we may end up with a different map.

\subsection{The scattering map for the perturbed geodesic flow}
\label{subsection:scatteringpert}
We now describe the scattering map for the system \eqref{eqn:hamrescaledeq}.
We stress that we use the notation $\,\tilde{}\,$ to denote the variables in
the extended system. That is, $\,\tilde{}\,$ refers to  adding the
extra variable  $\theta\in N$.

It is important to note that, if the dynamics in $N$ has
a small growth rate -- which  happens in our case for small
enough $\eps$ because the dynamics is given by the
vector field $\eps X$ -- then we can apply the theory of scattering map for non-autonomous systems recalled in Subsection \ref{subsection:scattering-general}.

As described in Subsection \ref{subsection:scaled}, $\tilde \Lambda_\eps = \Lambda_\eps  \times N$
is a normally hyperbolic invariant manifold for
the extended dynamics given by $\psi^\eps_{s}$.
Furthermore, $W^s(\tilde \Lambda_\eps ) = W^s(\Lambda_\eps) \times N$
is the stable manifold of $\tilde \Lambda_\eps $ under
the extended dynamics.
We also have that $W^s(z,\theta)  = W^s(z) \times \{\theta\}$.
Analogously for the unstable manifold.

Hence, if $\Gamma_\eps $ is a homoclinic  manifold satisfying
\eqref{scattering1} and \eqref{scattering2}, we see
that $\tilde \Gamma_\eps  = \Gamma_\eps  \times N$ will also
satisfy \eqref{scattering1} and \eqref{scattering2} in
the extended system.  The wave maps \eqref{eqn:wavemaps} associated to $\tilde \Gamma_\eps$ in the extended
system can be written as
\[\tilde\Omega^{\pm}_\eps(z, \theta) = (\Omega^{\pm}_\eps (z,\theta), \theta).\]

Therefore, if we can associate a scattering map
to $\Gamma_\eps $, we can also associate a scattering map
to $\tilde \Gamma_\eps $ and we have
\[\tilde S_\eps(z, \theta) = (S_\eps(z,\theta), \theta).\]
Note that the scattering map is the identity in the $N$ component.
For $\theta$ fixed, let $S_{\eps,\theta}$ denote the map given by $S_{\eps,\theta}(z)=S_\eps(z,\theta)$.

We can think of the scattering map either as the  mapping $\tilde S_\eps$, or
as a family of mappings $S_{\eps,\theta}$, indexed by the parameter $\theta$.
As proved in \cite{DLS06b}, the mappings $\tilde S_{\eps}$
and $S_{\eps,\theta}$ are smooth and depend
smoothly on parameters.

It is proved in \cite{DLS06b}, that
$S_{\eps,\theta}$ is symplectic as a  mapping on a domain in
$\Lambda_\eps$   if the flow is a time dependent symplectic flow.
(It is also proved in \cite{DLS06b} that $\Lambda_\eps$ is
a symplectic manifold.)

In our situation, one of the consequences of the smooth dependence on
parameters of $\tilde\Lambda_\eps$ and of its stable and unstable manifolds
is that we can find regions $U^\pm_0  \subset \Lambda_0$
which are independent of $\eps$,  for $0<\eps \le \eps_0 \ll 1$,
such that  $\tilde k_\eps(U^\pm_0\times N)\subseteq \tilde U^\pm_\eps$, where $\tilde k_\eps$ is the parametrization of $\tilde\Lambda_\eps$ described in Subsection~\ref{subsection:scaled}.

Via the  parametrization $\tilde k_\eps$, we can consider the scattering map $S_{\eps,\theta}$
as defined from $U^-_0$ to $U^+_0$, for $\theta\in N$.
That is, we can consider  the  scattering map as being defined between  domains
that are of product type and are independent of $\eps$ (of size of order $O(1)$).

\section{Elementary building blocks for the dynamics}
\label{sec:blocks}

In this section we construct some elementary building blocks of the dynamics. Each building block is a pseudo-orbit determined by one application of the scattering map followed by the application of the inner dynamics for some time.
The repeated construction of such  elementary  building blocks will produce  a two-dynamics pseudo-orbit which intersperses the
scattering map dynamics with the inner dynamics.

In Section \ref{section:constrwindows}, given  a sequence of elementary  building blocks,
we will  construct a sequence of windows which are correctly
aligned by the  dynamics. Afterwards, we will use the shadowing property of
correctly aligned windows from Subsection \ref{section:topological} to conclude the existence of a true orbit following
the sequence of elementary building blocks.

\subsection{The effect of the scattering map on the scaled energy} \label{subsection:scatteringaveraged}
The goal of this section
is to compute the change of the energy $H_\eps$ by one of the scattering maps when $\eps\in(0,\eps_0)$, for $\eps_0$ sufficiently small. Our goal will be to obtain estimates uniform in $\eps$.
The main observation is that the energy  is
a slow variable.

\subsubsection{Preliminaries}
We will be working with the scaled flow
\eqref{eqn:parham} in a fixed bounded range of scaled energies,
which we choose to be  $H_\eps\in[1,2]$. This determines a  compact subset $H_\eps^{-1}[1,2]\cap \tilde\Lambda_\eps$ in $\tilde\Lambda_\eps$.  Of course, obtaining estimates for all $\eps\in(0,\eps_0)$ corresponds to letting the  physical variables $(x,y)$ take values in a non-compact domain.

When we
will say that some error term is bounded by a constant (or by $O(\eps^a)$)
it will mean uniformly in that compact set.

Denote
\[
  \psi^\eps_{s}(\tilde z_\eps)
\equiv ( \geodesic^{\eps, q}_{s} (z_\eps), \geodesic^{\eps,p}_{s} (z_\eps),  \chi^{\eps}_{s}(\theta ))
\] the trajectory
of the   scaled Hamilton equations \eqref{eqn:hamrescaledeq}
with initial condition $\tilde z_\eps=(z_\eps,\theta)\in T^*M\times N$.

We have:
\begin{equation}\label{derivative-slow}
\begin{split}
\frac{d}{d s} H_\eps( \psi^\eps_{s}(\tilde z_\eps))
 &=
\eps^3\nabla_{\theta} V(\geodesic^{\eps,q}_{s} (z_\eps),  \chi^\eps_{s}(\theta ))\cdot X( \chi^\eps_{s}(\theta))\\
&\equiv
\eps^3 D_{X}V( \geodesic^{\eps,q}_{s} ( z_\eps),  \chi^\eps_{s}( \theta )).
\end{split}
\end{equation}

Here we regard $\frac{d}{ds} H_\eps$ as a functional
acting on the solution curves of the Hamiltonian equations \eqref{eqn:hamrescaledeq} in the extended phase space.

We will compute the leading term in $\eps$
of the change in energy by a scattering map in the perturbed equation
(see Subsection \ref{subsection:scatteringpert}).

In the rest of the subsection,
we will consider a fixed  homoclinic intersection $\tilde\Gamma_\eps$,
and we will denote by $\tilde S_\eps$ the corresponding scattering map.

The scattering map is defined from some domain
to some range in $\tilde\Lambda_\eps$.
Since the manifold $\tilde\Lambda_\eps$ depends on $\eps$, it is convenient
to reduce the scattering map  to some manifold which is independent  of $\eps$. We use that the manifold $\tilde\Lambda_\eps$
is diffeomorphic to the manifold $\tilde\Lambda_0$ via the parametrization $\tilde k_\eps$
mentioned in Subsection \ref{subsection:scaled}.
It is important to note that $\tilde k_\eps$ is $\eps$-close to the identity relative to the $C^r$-topology for $\eps$ small. We will use this fact below, when we compute the leading term in the expansion with respect to $\eps$ of the change of energy by the scattering map;
we will be able to approximate $\tilde k_\eps$  by the identity and incur only
error terms which are subdominant. Moreover,  $\tilde k_\eps$ can be chosen so that it is symplectic \cite{DLS06a}.

As $\Lambda_0$ is foliated by geodesics $\lambda_E$, $E\geq E_*$, we can parametrize
$\Lambda_0$ by the map $(E,s)\mapsto \lambda_{E}(s)$, where $E\geq E_*$ and the time $s$ is considered $\mod (1/\sqrt{2E})$,
i.e. $\lambda_{E}(s)=\lambda_{E}(s')$ if $s'-s\in (1/\sqrt{2E})\cdot\mathbb{Z}$.
Hence, we can parametrize  $\Lambda_\eps$ by $(E,s)\mapsto k_\eps(\lambda_{E}(s))$.

Thus, each points  $z_\eps\in \Lambda_\eps$ can be written as
$z_\eps= k_\eps(\lambda_{E}(s))$, for some unique ${E}$ and $s$.  Note that   $E$ and $s$  depend  on both the point
$z_\eps\in\Lambda_\eps$ and on the perturbation parameter $\eps$. Therefore, each point $\tilde z_\eps\in\tilde \Lambda_\eps$, with $\tilde z_\eps=(z_\eps,\theta)$ can be written as $\tilde z_\eps=\tilde k_\eps(\lambda_{E}(s))$, where $\tilde\lambda_{E}(s)=(\lambda_{E}(s),\theta)$.

Via the parametrization $\tilde k_\eps$, instead of $\tilde S_\eps$ we consider
the reduction $\tilde S^o_\eps=\tilde k_\eps^{-1} \circ \tilde S_\eps \circ \tilde k_\eps:\tilde k_\eps^{-1}(\tilde U_\eps^-)\to \tilde k_\eps^{-1}(\tilde U_\eps^+)$; note that the domain and codomain of $\tilde S^o_\eps$ are subsets  of  $\tilde \Lambda_0$. We note that when $\eps\to 0$ the scattering map $\tilde S^o_\eps$ approaches the unperturbed scattering map $S_0$ in the $C^{r-1}$ topology.

Since on $\Lambda_0$ we consider two coordinates systems, $(E,s)$ and $(J,\phi)$, we would like to make explicit the unperturbed scattering map in both coordinates. Given $S_0(z^-)=z^+$, in the action-angle coordinates, if $z^-=(J^-,\phi^-)$, $z^+=(J^+,\phi^+)$, then $J^-=J^+$ and $\phi^-+a=\phi^+$, and  in the energy-time coordinates, if $z^-=(E^-,s^-)$, $E^+=(E^+,s^+)$, then $E^-=E^+$ and $s^-+a/\sqrt{2E}=s^+$.

\subsubsection{The effect of the scattering map on the scaled energy} The first goal of this subsection is getting   quantitative
estimates on the change of scaled energy achieved by the scattering map:
\begin{equation} \label{Deltadefined}
\Delta(\tilde z^-_\eps) :=H_\eps(\tilde z^+_\eps)-H_\eps(\tilde z^-_\eps) = H_\eps(\tilde S_\eps(\tilde z^-_\eps))-H_\eps(\tilde z^-_\eps).
\end{equation}
We will write  each point $\tilde z_\eps^-$ as $(k_\eps(E,s),\theta)$, for some $E,s,\theta$.  We will express the leading term of the expansion of   $\Delta(\tilde z^-_\eps)$ with respect to $\eps$ in terms of the unperturbed system, and specifically in terms of the variables $(E,s,\theta)$.

\begin{prop}\label{prop:energyscattering} Let $\tilde z^-_\eps=(k_\eps(\lambda_{E}(s)),\theta)\in\tilde U_\eps^-$ and let $\tilde z_\eps^+=\tilde S_\eps (\tilde z_\eps^-)$. The change  of the scaled energy $H_\eps$ by the scattering map from $\tilde z_\eps^-$ to $\tilde z_\eps^+$ is given by
\begin{equation}\label{eqn:energyscattering}
\Delta(\tilde z^-_\eps)=\eps^3\Delta_1(E,s,\theta)+O(\eps|\ln\eps^4|),
\end{equation}
where  the leading term $\Delta_1$  in the expansion with respect to $\eps$ is given by
\begin{equation}\label{eqn:delta1}
\begin{split} \Delta_1(E,s,\theta)=  \lim_{T_\pm- \to \pm\infty} &\int_{T_-}^{T_+} \,
\left( D_XV \right )(\gamma^{q}_{E}(\sigma),\theta)d\sigma\\
&-\int_{0}^{T_+}\left( D_XV\right )( \lambda^{q}_{E}(\sigma+s+a/\sqrt{2E}),\theta)d\sigma\\
&-\int_{T_-}^{0}\left( D_XV\right )( \lambda^{q}_{E}(\sigma+s),\theta)d\sigma.\end{split}
\end{equation}
\end{prop}
\begin{proof}
Let $\tilde z_\eps=(\tilde\Omega^{-})^{-1}(\tilde z^-_\eps)=(\tilde\Omega^{+})^{-1}(\tilde z^+_\eps)\in\tilde \Gamma_\eps$. We start by noting that the orbit starting in $\tilde z_\eps$ is asymptotic in the future
to the orbit  of $\tilde z^+_\eps$ and in the past  to the orbit of $\tilde z^-_\eps$.
We can write
\begin{equation}
\begin{split}
\Delta(\tilde z^-_\eps) =
\lim_{T_\pm \to \pm \infty} &\left [
H_\eps(\psi^\eps_{T_+} (\tilde z_\eps)) -
H_\eps(\psi^\eps_{T_- } (\tilde z_\eps)) \right. \\
& - H_\eps(\psi^\eps_{T_+}(\tilde z^+_\eps)
+ H_\eps(\tilde z^+_\eps)  \\
& \left .+ H_\eps(\psi^\eps_{T_-}(\tilde z^-_\eps)
- H_\eps(\tilde z^-_\eps) \right]. \\
\end{split}
\end{equation}

By the Fundamental Theorem of Calculus, we have
\begin{equation}
\label{deltadefined}
\begin{split}
\Delta(\tilde z^-_\eps)
&=\lim_{{T_+\to\infty}\atop{T_-\to -\infty}}\left [\int_{T_-}^{T_+}
\left( \frac{d}{d \sigma}H_\eps \right )(\psi^\eps_{\sigma}(\tilde z_\eps))d\sigma\right. \\
&  -\int_{0}^{T_+}\left( \frac{d}{d\sigma} H_\eps\right)
(\psi^\eps_{\sigma}(\tilde z^+_\eps))d \sigma\\
&\left . -\int_{T_-}^{0}\left( \frac{dH_\eps}{d\sigma}\right)(\psi^\eps_{\sigma}(\tilde z^-_\eps))d\sigma\right ]\\
= &\lim_{T_+ \to \infty}\int_{0}^{T_+}\left[ \left( \frac{d}{d\sigma}H_\eps \right )(\psi^\eps_{\sigma}(\tilde z_\eps))
- \left( \frac{d}{d\sigma}H_\eps \right )(\psi^\eps_{\sigma}(\tilde z^+_\eps))\right]d\sigma\\
&+ \lim_{T_- \to - \infty}\int_{T_-}^{0}
\left[\left( \frac{d}{d\sigma}H_\eps \right )(\psi^\eps_{\sigma}(\tilde z_\eps))
- \left( \frac{d}{d\sigma}H_\eps \right )(\psi^\eps_{\sigma}(\tilde z^-_\eps))\right]d\sigma.\\
\end{split}
\end{equation}

It is important to remark that the integrands in the integrals
\eqref{deltadefined} converge exponentially fast. So do their
their derivatives of low order with respect to the initial conditions.
It is shown in \cite{DLS06b} that there is exponential
convergence of the integrand in \eqref{deltadefined} for the derivatives of an order lower than the ratio of
the  Lyapunov exponents in the stable and unstable directions
and the Lyapunov exponents tangent to the manifold.
In our case, since the directions
tangent to the manifold have zero exponent for $\eps = 0$, one can get
that the number of derivatives of the integrand that converge exponentially fast is arbitrarily large for $\eps$
small. In the arguments presented in this
paper, we will need just a moderate number of derivatives.

Our next goal is to prune \eqref{deltadefined} to extract a convenient
expression for the leading term.

By the exponential convergence of the integrands, we see that
for appropriately chosen constant $K> 0$, if
we take $T_+ = -T_- =  K|\ln(\eps)|$, the integrals
differ from the limit by not more than $O(\eps^4)$. Therefore,
\begin{equation}
\label{firststep}
\begin{split}
\Delta (\tilde z^-_\eps)&= \int_{- K | \ln(\eps)| }^{K |\ln(\eps)|}
\left( \frac{d}{d\sigma}H_\eps \right)( \psi^\eps_{\sigma}(\tilde z_\eps))d \sigma\\
&  -\int_{0}^{K |\ln(\eps)| }\left(
\frac{d}{d\sigma}H_\eps \right)
( \psi^\eps_{\sigma}(\tilde z^+_\eps))d\sigma\\
&  -\int_{-K| \ln(\eps)| }^{0}
\left( \frac{d}{d \sigma}H_\eps \right )( \psi^\eps_{ \sigma}(\tilde z^-_\eps))d \sigma + O(\eps^4).
\end{split}
\end{equation}

We now express $\Delta (\tilde z_\eps^-)$ in terms of the orbits of the unpertubed flow.
By the smooth
dependence on parameters of solutions of
ordinary differential equations, for $|\sigma| \le K | \ln(\eps)| $,
we have that the orbits of the perturbed system  are $O(\eps | \ln(\eps)|)$-close
to the corresponding orbits of the unperturbed flow.
Notice that, because $\eps | \ln(\eps)|$ is small, the
separation between the orbits is still growing linearly and
has not yet started to grow exponentially fast with time
\cite[estimate (3.5.4)]{Thirring}.

More precisely,  the orbits
$\geodesic^\eps_{\sigma}(\tilde z^\pm_\eps)$ are
$O(\eps |\ln(\eps)|)$-close
to the orbit $(\lambda_{E},\chi^\eps)$ where $\lambda_{E}$ is a  closed geodesic, and  the orbit $ \geodesic^\eps_{\sigma} (\tilde z_\eps)$ is
 $O(\eps | \ln(\eps)|)$-close
to $(\gamma_{E},\chi^\eps)$, where  $\gamma_{E}$ is a homoclinic orbit to $\lambda_{E}$ for the geodesic flow.
Also,  in these intervals of time, the variable
$\theta$ changes only by $O(\eps|\ln(\eps)|)$.
Thus,  for $|\sigma| \le K |\ln(\eps)| $, we have
\[
\begin{split}
& d(  \psi^\eps_{\sigma}(\tilde z_\eps), \tilde \gamma_{E}(\sigma) )
\le C\eps|\ln(\eps)| \\
& d(  \psi^\eps_{\sigma}(\tilde z^-_\eps), \tilde \lambda_{E}(\sigma+s))
\le C\eps|\ln(\eps)|,\\
& d(  \psi^\eps_{\sigma}(\tilde z^+_\eps), \tilde \lambda_{E}(\sigma+s+a/\sqrt{2E}))
\le C\eps|\ln(\eps)|,\\
\end{split}
\]
for some constant $C>0$, where $\tilde \gamma_{E}(\sigma) =(\gamma_{E}(\sigma),\theta)$ and
$\tilde \lambda_{E}(\sigma) =(\lambda_{E}(\sigma),\theta)$.

Therefore, substituting in the integral in \eqref{firststep}  the orbits of
the geodesic flow instead of the orbits of the perturbed flow and keeping $\theta$ constant,
 we incur an error $O(\eps^4|\ln(\eps)|)$. Using also the notation $D_XV$ introduced in
\eqref{derivative-slow}, we obtain:

\begin{equation}
\begin{split}
\Delta(\tilde z^-_\eps)&=
\int_{- K | \ln(\eps)| }^{K |\ln(\eps)|}
\left( \frac{d}{d\sigma} H_\eps \right)(\tilde \gamma_{E} (\sigma) )d\sigma\\
&  -\int_{0}^{K |\ln(\eps)| }\left( \frac{d}{d\sigma}H_\eps\right)
(\tilde \lambda_{E} (\sigma+s+a/\sqrt{2E}) ))d\sigma\\
& -\int_{-K | \ln(\eps)|}^{0}\left( \frac{d}{dt} H_\eps \right)
(\tilde \lambda_{E}(\sigma+s) )d\sigma +  O(\eps^4 |\ln(\eps)|) \\
&=\eps^3
\int_{- K | \ln(\eps)| }^{K |\ln(\eps)|}
 \left( D_XV \right)( \gamma^{q} _{E}(\sigma),  \theta )d\sigma\\
&  - \eps^3 \int_{0}^{K |\ln(\eps)| } \left( D_XV  \right)
( \lambda^{q} _{E}(\sigma+s+a/\sqrt{2E}),  \theta)) d\sigma\\
& -\eps^3 \int_{-K | \ln(\eps)|}^{0}
\left( D_XV \right)(  \lambda^{q} _{E}(\sigma+s),  \theta)d\sigma +  O(\eps^4 |\ln(\eps)|),
\end{split}
\end{equation}
where $ \lambda^{q}_{E}, \gamma^{q}_{E}$ denote the $q$-components of $\lambda _{E}, \gamma _{E}$, respectively.

Finally, because of the exponentially fast convergence
we can change the integral  over a $O(\ln(\eps))$ time interval
as $t\to\infty$ incurring an error $O(\eps^4)$, so
we obtain:
\begin{equation}
\label{finalcondition}
\begin{split}
\Delta(\tilde z^-_\eps)=&
\eps^3\lim_{{T_+\to\infty}\atop{T_-\to\infty}}\left[
\int_{T_-}^{T_+}
\left( D_XV\right )(\gamma^{q}_{E}(\sigma),\theta)d\sigma\right .\\
&\quad  -\int_{0}^{T_+} \left( D_XV \right)(\lambda^{q}_{E}(\sigma+s+a/\sqrt{2E}),\theta)d\sigma\\
&\quad  \left . -\int_{T_-}^{0}
\left(D_XV \right)(\lambda^{q}_{E}(\sigma+s),\theta)d\sigma\right ] \\
&\quad+O(\eps^4|\ln\eps|)\\
= & \eps^3 \left[
\lim_{T_+ \to \infty} \int_0^{T_+}  \,
\left( D_XV \right )(\gamma^{q}_{E}(\sigma),\theta)
- \left(D_XV \right )(\lambda^{q}_{E}(\sigma+s+a/\sqrt{2E}),\theta)d\sigma \right.  \\
&\phantom{ = \eps^3 \big[}
+\left.\lim_{T_- \to -\infty} \int_{T_-}^{0} \,
\left( D_XV \right )(\gamma^{q}_{E}(\sigma),\theta)
- \left( D_XV\right )( \lambda^{q}_{E}(\sigma+s),\theta)d\sigma\right]
\\&+O(\eps^4|\ln\eps|)\\
=&
\eps^3\Delta_1(E,s,\theta)
+ O( \eps^4 | \ln(\eps)|).
\end{split}
\end{equation}
In the last line we have just defined $\Delta_1=\Delta_1(E,s,\theta)$
as the leading term of $\Delta=\Delta(\tilde z^-_\eps)$, where $\tilde z^-_\eps=(k_\eps(\lambda_{E}(s)),\theta)\in \tilde U^-_\eps$.
\end{proof}

Notice that in the above argument we used the exponential convergence
of the integrands to justify the change of the limits, which  is often
done in Melnikov theory.

It is important to realize that the expression for  the leading term
$\Delta_1$ is in terms of  the unperturbed trajectories
and depends only on the perturbing potential.
It can be considered as a global Melnikov function. In contrast with
many standard treatments in which the Melnikov function is only defined
for periodic or quasi-periodic orbits, \eqref{finalcondition} is well defined for all
orbits in the domain of the scattering map independently of what is
their dynamics. Also, note that the function $\Delta_1$ can be viewed as an  analogue of what was called the
reduced Poincar\'e function in \cite{DLS00,DLS06a}.

\subsubsection{The effect of the scattering map on the action-angle coordinates}
Using action angle coordinates, we write the geodesic $\lambda_{E}(s)$
as $\lambda_{J^2/2}(s)$ and the  homoclinic
$\gamma_{E}(s)$ as $\gamma_{J^2/2}(s)$, where $E=J^2/2$.
By the  rescaling property  of the geodesic flow \eqref{eq:scale2} we have that
$\lambda_{J^2/2}(s)= \lambda_{1}(J s)$ and  $\gamma_{J^2/2}(s)=
\gamma_{1}(J s)$. Note that for an energy $E=1$ the corresponding action is $J=\sqrt{2}$. By the change of variable formula we obtain the
following

\[\begin{split}  \Delta_1(J,\phi,\theta)=&
\lim_{T_+ \to \infty} \int_0^{T_+}
[(D_XV)(\gamma^q_{J^2/2}(\sigma),\theta)-(D_XV)(\lambda^q_{J^2/2}(\sigma+\phi+a),\theta)]d\sigma\\
&+\lim_{T_- \to -\infty} \int_{T_-}^{0}
[(D_XV)(\gamma^q_{J^2/2}(\sigma),\theta)-(D_XV)(\lambda^q_{J^2/2}(\sigma+\phi),\theta)]d\sigma\\
&=\lim_{T_+ \to \infty} \int_0^{T_+}  [(D_XV)(\gamma^q_{1}(J \sigma),\theta)-(D_XV)(\lambda^q_{1}(J(
\sigma+\phi+a)),\theta)]d\sigma\\
&+\lim_{T_- \to -\infty} \int_{T_-}^{0}
[(D_XV)(\gamma^q_{1}(J \sigma),\theta)-(D_XV)(\lambda^q_{1}(J (\sigma+\phi)),\theta)]d\sigma\\
&=\frac{1}{J}\Delta_1(\sqrt{2},J\phi,\theta).
\end{split}\]
Thus, we have the following rescaling property of $\Delta_1$ relative
to action-angle coordinates:
\begin{equation}\label{scale-delta1}
\Delta_1(J, \phi,  \theta) = \frac{1}{J}\Delta_1(\sqrt{2}, J\phi,  \theta).
\end{equation}

\subsection{Change of energy over an elementary building block of a pseudo-orbit. }\label{subsec:buildingblock}

The goal of this section is to compute the change of energy over an elementary
building block of a pseudo-orbit, which consists in applying the scattering map followed by applying the inner dynamics
for some prescribed time. We will also formulate later some conditions that ensure
that the effect on the energy is non-trivial.  Of course,
this requires some choices (e.g., the time  we decide to follow the inner dynamics),  and, given the choices made,
some non-degeneracy assumptions
on the perturbations (e.g., a perturbation that vanishes identically will
not produce any effect).

We first compute the energy change over an elementary building block  obtained by starting at a point $\tilde z^-_\eps=(z_\eps,\theta)$,   applying  the scattering map, and then applying the inner dynamics for some time. We specify the  time for which we apply the inner dynamics implicitly by requiring  that the change of angle coordinate along the pseudo-orbit starting from $\tilde z^-_\eps$ is some fixed number $L$, chosen sufficiently large, to be specified later. We denote such an elementary building block by $B(\tilde z_\eps^-)$.

As specified before, we focus on a bounded energy range $H_\eps\in[1,2]$.  From Subsection \ref{subsection:scatteringaveraged}, the leading term in the energy change \eqref{finalcondition} depends on the effect of the perturbing potential on the unperturbed trajectories. Let $\tilde z^-_\eps=(k_\eps(\lambda_{E}(s)),\theta)$, with $E$ in the energy range.
Assume that the angle-action coordinates of $\lambda_{E}(s)\in \Lambda_0$ are $(J,\phi)$, with $J=\sqrt{2E}$.

By \eqref{eqn:scatteringcoords} the effect of the unperturbed  scattering map on $\Lambda_0$, taking a point $z^-$ to $z^+$, is to increase the angle coordinate $\phi$ by $a$. The scaled time $s$ to follow the inner dynamics starting from $z^+$ and ending at a point of angle coordinate equal to $L$ is  $(L - a)/\sqrt{2E}$. Hence, in terms of the unperturbed system, one follows the geodesic flow trajectory $\lambda_{E}(s)$ for the time interval $s\in[a/\sqrt{2E},L/\sqrt{2E}]$.

In the perturbed system, we choose to follow the inner dynamics for the same time interval, which
is  independent of $\eps$. Since the energy is a  slow variable and we are considering only
scaled times of order $1$,   the change of energy during the
time spent along  the inner dynamics can be computed, with a very small error, using the fundamental theorem
of calculus.

This implies that the change of energy along an orbit segment starting at some point $\tilde z^+_\eps=(z^+_\eps,\theta^+)$, where $\theta^+=\chi^\eps_{a/\sqrt{2E}}(\theta)$,   and following it for a time interval $s\in[a/\sqrt{2E},L/\sqrt{2E}]$ is
\begin{equation*}\begin{split}\int_{a/\sqrt{2E}}^{L/\sqrt{2E}} \frac{d}{dt}H_\eps(\geodesic^{\eps}_{\sigma}( z^+_\eps), \chi^\eps _{\sigma}( \theta))d\sigma=\eps^3\int_{a/\sqrt{2E}}^{L/\sqrt{2E}} \left( D_XV \right)( \lambda^{q}_{E}(\sigma), \theta) d\sigma\\+O(\eps^4|\ln\eps|).\end{split}\end{equation*}

\begin{prop}\label{prop:energyblock}
Let $\tilde z^-_\eps=(k_\eps(\lambda_{E}(s),\theta)$, and $\tilde S_\eps(\tilde z^-_\eps)=\tilde z^+_\eps$.
Consider an elementary building block consisting of one application of the scattering map $\tilde S_\eps(\tilde z^-_\eps)=\tilde z^+_\eps$, and a trajectory segment  with initial point $\tilde z^+_\eps$ following the inner dynamics  for a  time interval $s\in[a/\sqrt{2E},L/\sqrt{2E}]$.
 The change $G(\tilde z_\eps^-)$ of scaled energy $H_\eps$ over the building block  is of the form
\begin{equation}\label{G}
     G(\tilde z_\eps^-)  =\eps^3  G_1(E,p,s) + O(\eps^4|\ln\eps|),
\end{equation}
where $G_1$, the leading term of $G$, is given by
\begin{equation}
\label{G1}\begin{split}
G_1(E,p,s)&=\lim_{{T_+\to\infty}\atop{T_-\to\infty}} \left [\int_{T_-}^{T_+}
\left( D_XV  \right )( \gamma^{ q}_{E}(s), \theta)d\sigma\right .\\
& -\int_{0}^{ T_+}\left(D_XV  \right)( \lambda^{\ q}_{E}(\sigma+s+a/\sqrt{2E}), \theta)d\sigma\\
&-\int_{T_-}^{0}\left(D_XV \right)( \lambda^{ q}_{E}(\sigma+s), \theta)d\sigma
\\
 &\left. + \int_{a/\sqrt{2E}}^{L/\sqrt{2E}}
\left(D_XV  \right)( \lambda^{q}_{E}(\sigma), \theta) d\sigma
\right ].
\end{split}
\end{equation}
\end{prop}

\begin{proof}
By the fundamental theorem of calculus, the change of energy along an orbit segment starting at  $\tilde z^+_\eps$ and following it for a time interval $s\in[a/\sqrt{2E},L/\sqrt{2E}]$ is
\begin{equation*}\begin{split}\int_{a/\sqrt{2E}}^{L/\sqrt{2E}} \frac{d}{dt}H_\eps(\geodesic^{\eps}_{s}( z^+_\eps), \chi^\eps _{s}( \theta))ds=\eps^3\int_{a/\sqrt{2E}}^{L/\sqrt{2E}} \left( D_XV \right)( \lambda^{ q}_{E}(s),\theta) ds\\+O(\eps^4|\ln\eps|).\end{split}\end{equation*}

Combining with \eqref{finalcondition},
  we see that, over an elementary building block, the
energy has changed by
\begin{equation*}
\label{gain}
\begin{split}
G(\tilde z_\eps^-) & = \Delta(\tilde z_\eps^-) +\int_{a/\sqrt{2E}}^{L/\sqrt{2E}} \frac{d}{dt}H_\eps(\geodesic^{\eps}_{s}( z^+_\eps), \chi^\eps _{s}( \theta))ds \\
& =
\eps^3\left[\Delta_1(\tilde z_0^-)+
 \int_{a/\sqrt{2E}}^{L/\sqrt{2E}}
\left(D_XV\right)( \lambda^{ q}_{E}(s), \theta) ds
\right ] \\
& +O(\eps^4|\ln\eps|).
\end{split}
\end{equation*}

In conclusion,
\begin{equation*}
G(\tilde z_\eps^-)  =\eps^3G_1(E,s,\theta) + O(\eps^4|\ln\eps|).
\end{equation*}
\end{proof}

Similarly to \eqref{scale-delta1}, we have the following rescaling
property of $G_1$ relative to action-angle coordinates
\begin{equation}\label{scale-G1}
 G_1(J, \phi,  \theta) = \frac{1}{J} G_1(\sqrt{2}, J\phi,  \theta).
\end{equation}

We make two important remarks.

\begin{rem} For the scattering map $\tilde S_\eps(\tilde z^-_\eps)=\tilde z^+_\eps$,  there is no trajectory of the system asymptotic to  $\tilde z^-_\eps$ in the past and to $\tilde z^+_\eps$ in the future. Rather,   $\tilde\psi_{T-}^\eps(\tilde z_\eps)$ will approach  $\tilde\psi_{T-}^\eps(\tilde z^-_\eps)$ as $T_-\to -\infty$, and $\tilde\psi_{T+}^\eps(\tilde z_\eps)$ will approach $\tilde\psi_{T+}^\eps(\tilde z^+_\eps)$ as $T_+\to +\infty$. Here $\tilde z_\eps=(\tilde\Omega^{-})^{-1}(\tilde z^-_\eps)=(\tilde\Omega^{+})^{-1}(\tilde z^+_\eps)$. For $-T_-=T_+=K|\ln(\eps)|$, the change of energy $\Delta(\tilde z^-_\eps)$ along the homoclinic trajectory from
$\tilde\psi_{T-}^\eps(\tilde z_\eps)$ to $\tilde\psi_{T+}^\eps(\tilde z_\eps)$ is the same as in \eqref{finalcondition}, up to an error term which is subdominant.\end{rem}

\begin{rem} Similarly, if instead of fixing the angle shift to a constant value $L$, we allow to chose, for each elementary building block,  a value of $L$ which is constant plus an  $O(\eps)$-term, the change of energy $G(\tilde z_\eps^-)$ is the same as in \eqref{G},
up to an error term  which is subdominant.\end{rem}

In Section \ref{section:constrwindows} we will show that there is a trajectory of the dynamics that follows closely the pseudo-orbit consisting of the segment of the  homoclinic trajectory from  $\tilde\psi_{T-}^\eps(\tilde z_\eps)$ to $\tilde\psi_{T+}^\eps(\tilde z_\eps)$,   followed by a segment of the trajectory of the inner flow $(\tilde\psi^\eps_{s})_{\mid\tilde\Lambda_\eps}$  with initial point $\tilde\psi_{T+}^\eps(\tilde z_\eps^+)$.

Based on these remarks, we will keep in mind that to an elementary building block we can  associate the following objects:
\begin{itemize}
\item One application of the scattering map $\tilde S_\eps(\tilde z^-_\eps)=\tilde z^+_\eps$, plus one segment of the trajectory of the inner flow $(\tilde\psi^\eps_{s})_{\mid\tilde\Lambda_\eps}$ with initial point $\tilde z_\eps^+$;
\item A pseudo-orbit, consisting of a segment of a homoclinic orbit from  $\tilde\psi_{T-}^\eps(\tilde z_\eps)$ to $\tilde\psi_{T+}^\eps(\tilde z_\eps)$, followed by a segment of the trajectory of the inner flow $(\tilde\psi^\eps_{s})_{\mid\tilde\Lambda_\eps}$  with initial point $\tilde\psi_{T+}^\eps(\tilde z_\eps^+)$;
\item A true orbit, that follows closely the pseudo-orbit described above.
\end{itemize}
The change of energy along either one of these objects is given by the estimate in Proposition \ref{prop:energyblock}, up to a subdominant error term.

\subsection{Generic set of potentials}\label{subsection:generic}

In this section we specify the set of potentials $\mathcal{V}'$ claimed in Theorem \ref{thm:main1}. The potential $V\in\mathcal{V}'$ are required to satisfy a condition that ensures that one can achieve consistent energy growth by applying the scattering map followed by the inner dynamics as in Proposition \ref{prop:energyblock}.

A key observation is that, since there exist a closed hyperbolic geodesic $\lambda_E$ and a corresponding transverse homoclinic orbit $\gamma_E$, by the Birkhoff-Smale Homoclinic Orbit Theorem (see e.g., \cite{KatokH1995}) there exist, in fact, at least two geometrically distinct homoclinic orbits to the same   geodesic. The existence of at least two homoclinic orbits can also be argued via variational methods.
We shall denote a pair of such homoclinic orbits by $\gamma^1_E$, $\gamma^2_E$, and the associated scattering maps $\tilde S^1_\eps$, $\tilde S^2_\eps$. We  denote the corresponding  leading terms $G_1$ from Proposition \ref{prop:energyblock} by $G^1_1,G^2_1$, respectively. The main idea is that, under some generic condition  on the potential $V$, utilizing one of the homoclinic orbits to grow energy will be more advantageous than utilizing the other, so one can select which one of the two homoclinic orbits to use in a way to ensure a net gain of energy over time.

To express the generic condition on the potentials $V\in\mathcal{V}'$, we write $G^1_1,G^2_1$ in action-angle coordinates. Indeed,  the energy growth \eqref{G} along an elementary building block depends only on the initial point $\tilde z_\eps^-=(k_\eps(\lambda_{J^2/2}(\phi)),\theta)\in\tilde\Lambda_\eps$ of the block, and  the leading term $G_1$ in \eqref{G1} depends on the corresponding angle-action coordinates and parameter value  $(J,\phi,\theta)$ corresponding to   $\tilde z_\eps^-$.

We define $\mathcal{V}'$ to be the set of potentials $V$ for which the following non-degeneracy condition holds:

\textsl {A4. Fix $ J_0 =\sqrt{2}$ so the corresponding energy level of $H_0$ is $E_0=1$. Fix $\theta_0\in N$ a non-trivial uniformly recurrent point.   We assume that for geodesic flow at this energy level there exist   two  geometrically different homoclinic trajectories $\gamma^1_{E_0}$, $\gamma^2_{E_0}$ to the same closed geodesic $\lambda_{E_0}$, and a common domain $\tilde U$ for the corresponding scattering maps $\tilde S^1_\eps,\tilde S^2_\eps$ such that:
\begin{equation}\label{eqn:A4} \sup_{\phi_1}G^1_1(J_0,\phi_1,\theta_0)\neq \sup_{\phi_2}G^2_1(J_0,\phi_2,\theta_0),\end{equation}
where  $G^1_1,G^2_1$ are the leading terms of the energy gain, given by \eqref{G1}.}

In the above, it is understood that the angles $\phi_1$, $\phi_2$ are restricted  to some closed intervals
where $(\lambda_{J_0^2/2}(\phi_1),\theta_0)\in \tilde U$, $(\lambda_{J_0^2/2}(\phi_2),\theta_0)\in \tilde U$,
where $\tilde k_\eps(\tilde U)$ is the domain $(\tilde U^1_\eps)^-$ of $\tilde S^1_\eps$ and also in the domain $(\tilde U^2_\eps)^-$
of $\tilde S^2_\eps$. Of course, a domain $\tilde U$ as in condition \textit{A4} is not unique. The condition requires only the existence of at least one domain $\tilde U$ on which \eqref{eqn:A4} holds.

We notice that because we assumed  the closed geodesics in the unit tangent bundle
have transverse homoclinic connections for the geodesic flow
in the unit tangent bundle, we can define the projections along each of
the points on the orbit. Similarly, we can lift for any value of
the energy by the scaling invariance.

Therefore, we can always define locally two scattering maps and we can continue them.
The only obstruction to define a scattering map in an arbitrary domain is that the local
continuation along a closed loop may have some monodromy.
Each of the scattering maps can be defined on any domain that does not contain essential circles, i.e., non-contractible loops of the cylinder. Hence, the assumption in \textit{A4} that the two scattering maps have a common domain is satisfied automatically.

Condition \textit{A4} is a condition on $V$  along trajectories of the unperturbed system, as shown by \eqref{G1}. This is an explicit condition, that can be verified for a given $V$ along a given closed geodesic and a given pair of homoclinic orbits  to that closed geodesic.

Assuming condition \textit{A4}, suppose, without loss of generality,   that   \[\sup_{\phi_1}G^1_1(J_0,\phi_1,\theta_0)> \sup_{\phi_2}G^2_1(J_0,\phi_1,\theta_0).\] Then there exists $\delta>0$, $\phi_*$  and an open neighborhood $\mathcal{P}\subseteq N$ of $\theta_0$ such that
\begin{equation}
\label{2G}G^1_1(J_0,\phi_*,\theta)-G^2_1(J_0,\phi,\theta)\geq 2\delta
\end{equation} for all $\theta\in\mathcal{P}$ and  all   $\phi$ with $\tilde k_\eps(\lambda_{E_0}(\phi),\theta_0)$ is in the domain of $\tilde S^2_\eps$. Moreover, we can choose the neighborhood $\mathcal{P}$ to be a flow box for the flow $\chi^\eps$ on $N$, i.e., a homeomorphic copy in $N$ of a $d$-dimensional open rectangle,  of the form $\{\chi^\eps_s(\theta)\,|\, \theta\in \Sigma,s\in(-\rho,\rho)\}$, where $\Sigma$ is a $(d-1)$-dimensional open disk transverse  to the flow $\chi^\eps_s$, and $\rho>0$ is sufficiently small so that the flow $\chi^\eps_s$ is transverse to each surface $\{\chi^\eps_{s_0}(\theta)\,|\,\theta\in \Sigma\}$ for each $s_0\in(-\rho,\rho)$.

The condition (A4) is formulated in terms of the geodesic and
homoclinic orbits at some fixed energy level $E_0=1$, corresponding to
$J_0=\sqrt{2}$. By the rescaling property of the geodesic flow, and by
the corresponding rescaling property \eqref{scale-G1}, it follows that
\begin{equation}\label{scale-G1J}G^1_1(J,\frac{1}{J}\phi_*,\theta)-G^2_1(J,\phi',\theta)\geq
\frac{2\delta }{J},\end{equation}
for all $\theta \in\mathcal{P}$ and all $\phi'=\frac{1}{J}\phi$ with $\phi$ restricted as before.
Since we restrict to the interval
$H_\eps\in[1,2]$, we have $\sqrt{2}\leq J\leq 2$ hence
\[G^1_1(J,\frac{1}{J}\phi_*,\theta)-G^2_1(J,\phi',\theta)\geq \delta,\]
for all $\theta \in\mathcal{P}$ and all $\phi'$ as before.

\begin{lem}\label{lemma:generic}
Given a  geodesic flow, a closed geodesic  $\lambda$  and two geometrically distinct  homoclinic
orbits  $\gamma^1_{E_0}$, $\gamma^2_{E_0}$,  at the energy level $E_0=1$.
Then,  the set $\mathcal{V}'$ of potentials $V$  that satisfy assumption (A4) is $C^0$ open and $C^\infty$ dense in the set of all the potentials.
\end{lem}

\begin{proof}
We assume by contradiction that, for some domain $\tilde U$, the condition (A4) does not hold.

Let $\sup_{\phi_1}G_1^1 (J_0,\phi_1,\theta_0) =G_1^1 (J_0,\phi_1^*,\theta_0)$ for some $\phi_1^*$, and
$\sup_{\phi_2}G_1^2 (J_0,\phi_2,\theta_0) =G_1^2 (J_0,\phi_2^*,\theta_0)$ for some $\phi_2^*$.
Then  we have
\begin{equation} \label{negation}
\begin{split}G_1^1 (J_0,\phi_1^*,\theta_0) = G_1^2(J_0,\phi_2^*,\theta_0).
\end{split}
\end{equation}

Since for fixed $(J_0,\phi_1^*)$ and  $(J_0,\phi_2^*)$  the functions
$G_1^1(J_0,\phi_1^*,\theta_0)$ and $G_1^2(J_0,\phi_2^*,\theta_0)$, respectively,   considered as  functionals of $V$,
are continuous when the set $\mathcal{V}$ of potentials is given the $C^0$ topology, it is clear that
\eqref{negation} defines a $C^0$-closed set (intersection of
closed sets) and, therefore the condition (A4) holds in a $C^0$-open set of potentials.

To complete the proof of Lemma \ref{lemma:generic}, it suffices to
show that, given a potential $V$ that satisfies \eqref{negation}, there is
a  arbitrarily $C^\infty$-small perturbation of $V$ which does not
satisfy \eqref{negation}.

The construction is very clear. We note that $G^j_1$, $j\in\{1,2\}$ is a sum of  integrals over several trajectory segments  of the geodesic flow:
some are segments of closed geodesics in $\Lambda_0$, which are recurrent, and one of them is a segment of a  homoclinic trajectory.

Because the two homoclinic orbits $\gamma^1_{E_0}, \gamma^2_{E_0}$ are geometrically different,
we can find a $s_0 \in \mathbb{R}$ and a small enough ball $B \subset T^*M$  centered
around $\gamma^1_{E_0}(s_0)$ in such a way that
$B \cap \lambda_{E_0}= B \cap \gamma^2_{E_0} =\emptyset$.
Moreover, we can choose the ball $B$ such that there  exists a  small interval
$I$ around $s_0$ such that
$\gamma^1(\mathbb{R}) \cap B = \gamma^1(I)$. Choose $B'$ a small ball in $N$ centered around $\theta_0$.

Now we choose a function $W:T^*M \times N \rightarrow \mathbb{R} $
with support in $B \times B'$ so that
$\nabla_\theta W( \gamma^1_{E_0}(s), \theta_0)\cdot X(\theta_0) \ge  \rho > 0$ for all $s\in I$ and $\theta\in B'$, and for some $\rho>0$.
It follows from \eqref{eqn:delta1} that perturbing $V$ to  $V+W$ yields $G^1_1(J_0,\phi_1^*,\theta_0)\neq  G^2_1(J_0,\phi_2^*,\theta_0)$.
\end{proof}

\subsection{Gain of energy along sequences of elementary building blocks.}\label{subsection:gainseq}

We will estimate the gain of the scaled energy $H_\eps$ in time,   along some suitably chosen  sequences of elementary building blocks, for potentials $V\in\mathcal{V}'$. We will account for both the scaled time $s$  and the physical time $t$.

Assume that $V$ satisfies condition (A4) from the previous section. Then for given any $J_0$ with $\|J_0\|=\sqrt{2}$ there exist $\phi_*$ with $\tilde k_\eps(\lambda_{E_0}(\phi_*),\theta_0)\in  (\tilde U^1_\eps)^-$ and a flow box  $\mathcal{P}\subseteq N$, of size $O(1)$, with the property that,
if $\theta\in\mathcal{P}$ and $\tilde k_\eps(\lambda_{E_0}(\phi),\theta_0)\in  (\tilde U^2_\eps)^-$,  then
\begin{equation*} G_1^1(J_0 ,\phi_*,\theta)-G^2_1(J_0 ,\phi,\theta)\geq 2\delta.
\end{equation*}

Due to the rescaling property of the geodesic flow, it  follows that for any $J\in[\sqrt{2},2]$ (corresponding to the energy range $E\in[1,2]$ fixed in Subsection \ref{subsec:buildingblock}), there exists $\phi_*(J)=\phi/J_*$ such that
\begin{equation}\label{eqn:G1ineq}G_1^1(J,\phi_*(J),\theta)-G^2_1(J ,\phi,\theta)\geq \delta,
\end{equation}
for all $\theta\in\mathcal{P}$ and all $\phi$ in the appropriate domain.

The flow $\chi^\eps_s$ is slow, and so is its time-$1$ map. It takes a  time $O(1/\eps)$ to travel a distance $O(1)$.
The flow box $\mathcal{P}$ was chosen of the form $\{\chi^\eps_s(\theta)\,|\, \theta\in \Sigma,s\in(-\rho,\rho)\}$ with  $\Sigma$ as in Subsection \ref{subsection:generic}.

Assume that $\theta_0\in N$ is a uniformly recurrent point with $X(\theta_0)\neq 0$. Choosing the neighborhood $\mathcal{P}$ small enough ensures that the trajectory of $\theta_0$ will successively leave $\mathcal{P}$ and return to $\mathcal{P}$.
The lengths of the time intervals when the trajectory of $\theta_0$ moves through $\mathcal{P}$ are uniformly bounded above and below, and, because of the uniform recurrence hypothesis \textsl{A3}, so are the lengths of the time intervals when the  trajectory of $\theta_0$ moves through $N\setminus\mathcal{P}$. More precisely, there
exist $0<\tau_0<\tau'_0$, independent of $\eps$, such that the trajectory of $\theta_0$ spends a scaled time between $\tau_0/\eps$ and $\tau'_0/\eps$ in $\mathcal{P}$, and there
exist $0<\tau_1<\tau'_1$, independent of $\eps$, such that the trajectory of $\theta_0$ spends a scaled time between $\tau_1/\eps$ and $\tau'_1/\eps$ in $N\setminus\mathcal{P}$ between successive returns to $\mathcal{P}$.

Now we compute the growth of energy during a range of scaled time $\Delta s=1/\eps^2$. We follow a sequence of elementary building blocks of the type $B^1(J,\phi,\theta)$,  $B^2(J,\phi,\theta)$, where the superscripts correspond to the two choices of homoclinic orbits/scattering maps $\gamma^1_{E},\gamma^2_{E}$ respectively, where the succession of blocks is chosen as follows.   When $\theta \in \mathcal{P}$ we use blocks of the type $B^1(J,\phi_*(J),\theta)$, where $\phi_*(J)$ is defined as before.  When $\theta\not\in\textrm{cl}( \mathcal{P})$ we use blocks of the type  $B^2(J,\phi_0,\theta)$ for some $\phi_0$ fixed. Thus, the sequence is composed of strings of the type  $B^1(J,\phi_*(J),\theta)$, alternating with strings of the type $B^2(J,\phi_0,\theta)$;  the  proportion of time when we switch from $B^1(J,\phi_*(J),\theta)$ to $B^2(J,\phi_0,\theta)$ or viceversa is $O(\eps)$.

We  compute the  growth of energy along such a sequence of building block spanning a range of   scaled time  of  $1/\eps^2$. The initial condition for the dynamics   on $N$ is  the point $\theta_0$ which is assumed to be uniformly recurrent. The $J$-coordinate along the pseudo-orbit takes the successive values $J(n)$, and the $\theta$-coordinate along the pseudo-orbit takes the successive values $\theta(n)$; the $\phi$-coordinate is maintained fixed $\phi=\phi_0$.  

We have:

\begin{equation}\begin{split}\label{eqn:calculus1}
\Delta H_\eps=\eps^3& [\sum _{{n\in[0,1/\eps^2]}\atop{\theta(n)\in \mathcal{P}, \theta((n+1)\eps)\in \mathcal{P}}}
G^1_1(J(n),\phi_*(J(n)), \theta(n))  \\
& +\sum _{{n\in[0,1/\eps^2]}\atop{\theta(n)\not \in \mathcal{P}, \theta((n+1)\eps)\not\in \mathcal{P}}}
G^2_1(J(n),\phi_0, \theta( n)) \\
& + \sum _{{n\in[0,1/\eps^2]}\atop{\theta(n)\in \mathcal{P}, \theta((n+1)\eps)\not\in \mathcal{P}}}
  G^2_1(J(n),\phi_0, \theta( n))\\
& + \sum _{{n\in[0,1/\eps^2]}\atop{\theta(n)\not\in \mathcal{P}, \theta((n+1)\eps)\in \mathcal{P}}}
 G^2_1(J(n),\phi_0,\theta( n))]\\
& +O(\eps^2|\ln(\eps)|),\end{split}
\end{equation}
where the above error term is due to the accumulation  error term of $O(\eps^4|\ln(\eps)|)$ from \eqref{G1} over $1/\eps^2$ time steps.
The terms corresponding to the times $n$ when $\theta(n)\in \mathcal{P},\theta((n+1)\eps)\not\in \mathcal{P}$ and $\theta(n)\not\in \mathcal{P}, \theta((n+1)\eps)\in \mathcal{P}$ are $O(\eps \cdot 1/\eps^2)=O(1/\eps)$, and since $G^1_1$, $G^2_1$ are bounded, they contribute to a combined error
term $O(\eps^3\cdot 1/\eps)=O(\eps^2)$, which is subdominant.

Thus we can write
\begin{equation}\begin{split}\label{eqn:calculus2}
\Delta H_\eps&\geq \eps^3[\sum _{{n\in[0,1/\eps^2]}\atop{\theta(n)\in \mathcal{P}, \theta((n+1)\eps)\in \mathcal{P}}}
 G^1_1(J(n),\phi_0, \theta(n)) \\
& +  \sum _{{n\in[0,1/\eps^2]}\atop{\theta(n)\not \in \mathcal{P}, \theta((n+1)\eps)\not\in \mathcal{P}}}
G^2_1(J(n),\phi_0, \theta(n))]\\
& +O(\eps^2 |\ln(\eps)|).\end{split}
\end{equation}

Now we rearrange the summation above to estimate the total gain of energy while $\theta\in\mathcal{P}$:
\begin{equation}\begin{split}\label{eqn:calculus3}
\Delta H_\eps&\geq\eps^3 \sum _{{n\in[0,1/\eps^2]}\atop{\theta(n)\in \mathcal{P}, \theta((n+1)\eps)\in \mathcal{P}}}
[G^1_1(J(n),\phi_0, \theta(n))-G^2_1(J(n),\phi_0, \theta(n))]\\
&+\eps^3 \sum _{n\in[0,1/\eps^2]} G^2_1(J(n),\phi_0, \theta(n))\\
&+O(\eps^2|\ln(\eps)|)\\
&\geq \eps^3\delta \frac{\tau_0}{\eps^2} +\eps^3 \sum _{n\in[0,1/\eps^2]} G^2_1(J(n),\phi_0, \theta(n))) +O(\eps^2|\ln(\eps)|)\\
&=\eps\tau_0\delta+\eps^3 \sum _{n\in[0,1/\eps^2]} G^2_1(J(n),\phi_0, \theta(n))+O(\eps^2|\ln(\eps)|).
\end{split}
\end{equation}

Now we treat the remaining summation from above as a Riemann sum of mesh $1$ and we approximate it by a Riemman integral.
Since $J(\sigma)$ changes by  at most $O(\eps^2)$ and $\theta(\sigma)$ changes by at most $O(\eps)$ over each interval $[n,n+1]$, we have
$\max_{\sigma\in[n,n+1]} |\frac{d}{d\sigma} G_1(J(\sigma),\phi,\theta(\sigma))|=O(\eps)$ for all $n$.
Hence the error in approximating the integral on  a subdivision by a term of the Riemann sum is less than $O(\eps)$, and since there are  $1/\eps^2$ terms in the Riemann sum, the error in approximating the integral by the Riemann sum is $O(1/\eps)$. Taking into account the $\eps^3$-leading factor we obtain
\begin{equation}\label{eqn:calculus4}
\begin{split}
\eps^3\sum _{n\in[0,1/\eps^2]} G^2_1(J(n),\phi_0, \theta(n))\leq
\eps^3\int _{0}^{1/\eps^2} G^2_1(J(\sigma),\phi_0, \theta(\sigma))d\sigma + O(\eps^3\frac{1}{\eps})\\\leq
\eps^3[A^2_1(J(1/\eps^2),\phi_0, \theta(1/\eps^2))-A^2_1(J(0),\phi_0, \theta(0))]+ O(\eps^2)=O(\eps^2),
\end{split}
\end{equation} where
$G^2_1=   D_X A^2_1$, where  the antiderivative $A^2_1$ of $G^2_1$ is bounded by the compactness of $(\Lambda_0\times N)\cap H_0^{-1}[1,2]$. (Here we used the fact that $\frac{d}{dt}A^2_1(J(t),\phi_0,\theta(t))=\frac{dA^2_1}{dJ}\dot J+ D_XA^2_1 =
D_XA^2_1 +O(\eps^2)=G^2_1+O(\eps^2)$.)

From \eqref{eqn:calculus3} we conclude that, under the non-degeneracy assumption \textsl{A4}, the gain of energy following  a string of elementary building blocks of the type $B^1(J,\phi_0,\theta)$,  $B^2(J,\phi_0,\theta)$, over a range of scaled time $1/\eps^2$ satisfies
\begin{equation}\label{eqn:calculus5}
\begin{split}
\Delta H_\eps\geq \eps\tau_0\delta+O(\eps^2|\ln(\eps)|).
\end{split}
\end{equation}

Note that during this time interval we do not leave the scaled energy interval $E\in[1,2]$ fixed at the beginning of the argument.

Thus, during a time period of $1/\eps^2$, moving along the pseudo-orbits corresponding to the elementary building blocks  of the type $B^1(J,\phi_0,\theta)$,  $B^2(J,\phi_0,\theta)$, in the specified order, corresponds to a scaled energy growth of $O(\eps)$.
Since $H_\eps=\eps^2H_0$ and $\Delta s =\Delta t/\eps$, we obtain a physical energy growth of $\Delta H =O(1/\eps^2)$ during  a physical time interval $\Delta t=O(1/\eps^2)$, that is,  a linear growth rate of the physical energy in physical time.

\begin{rem}
If we construct a sequence of elementary building blocks corresponding to a single homoclinic $\gamma^j_{E_0}$, $j\in\{1,2\}$,  for a time of $O(1/\eps)$, and if the variable $\phi_0$ in the construction is fixed as above, a calculation as in \eqref{eqn:calculus4}  shows that the change of scaled energy is $O(\eps^3)$. This implies that $G^j_1$ cannot be always positive or always negative for all this time. Along any sequence of elementary building blocks for which the external flow $\chi^\eps$  returns to a small neighborhood of its initial point $\theta_0$,  there will always be regions in $\tilde \Lambda_0$ where $G^j_1$ is positive as well as regions where $G^j_1$ is negative. By the same reason, we cannot have $G^1_1>G^2_1$ for all time. Hence, besides a flow box $\mathcal{P}\subseteq N$  such that
$G^1_1(J_0,\phi_0,\theta)-G^2_1(J_0,\phi_0,\theta)>0$ for all $\theta\in\mathcal{P}$, there should also exist another flow box $\mathcal{P}'$ such that $G^2_1(J_0,\phi_0,\theta)-G^1_1(J_0,\phi_0,\theta)>0$ for all $\theta\in\mathcal{P}'$.
\end{rem}

\subsection{Sequences of elementary building blocks achieving unbounded growth of energy}\label{subsection:unbounded}
The above construction of pseudo-orbits can be continued for a time $O(1/\eps^3)$ to achieve an energy growth
corresponding to the whole interval $H_\eps\in[1,2]$.  Since $\eps=1/\sqrt{E_*}$, the corresponding growth of physical energy is $H=\frac{1}{\eps^2}H_\eps\in[E_*,2E_*]$.
To grow the physical energy to infinity, we repeat the procedure, re-initializing the process staring with $\eps=1/\sqrt{2E_*}$. For this new value of the small parameter $\eps$ in \eqref{eqn:hamrescaled}, growing the physical energy $H\in [2E_*,4E_*]$ amounts to growing the scaled energy $H_\eps\in[1,2]$. So all the estimates made in this subsection remain valid and carry forward. Thus, this construction of pseudo-orbits featuring energy growth can be repeated indefinitely.

In Section \ref{section:existenceoforbits} we show the existence of true orbits `shadowing' the pseudo-orbits constructed in this section.

\subsection{Sequences of elementary building blocks achieving symbolic dynamics}\label{subsection:symbolic}
In order to construct elementary building blocks along which the  energy follows a prescribed path $\mathcal{E}:[0,\infty)\to \mathbb{R}$,  we alternate elementary building blocks  leading to
 energy growth with blocks leading to energy loss.
Choosing the proportions of the energy growth and of the energy loss
allows us to control the energy change.  In particular, we
can obtain rates close to zero by alternating  blocks which gain energy
with blocks which loose energy. The change of energy along an elementary building block is
not more that $\eps^3$ in the scaled variables, which corresponds
to a change of energy of $E^{-1/2}$ in the physical variables. In this way, we can follow the prescribed energy path $\mathcal{E}$ up to $E^{-1/2}$.
Then, once the sequence of elementary building blocks is  constructed, in Section \ref{section:constrwindows}
we will  construct a sequence of correctly aligned windows along this sequence of block, and apply the
shadowing Theorem \ref{theorem:detorb} and obtain an  orbit that follows these windows.

\begin{rem} We remark that in the present mechanism we  do not achieve
small rates of growth by staying near a KAM torus, as in \cite{DLS00,DLS06a}, rather the orbits we construct are performing homoclinic excursions most of the time. Thus, these orbits are very different from the previously constructed orbits.
\end{rem}

\section{Existence of orbits following sequences of elementary building blocks}\label{section:existenceoforbits}

In this section, we   show that  we can concatenate infinitely many elementary building blocks as above, and that there exist a true orbit that follows the pseudo-orbit underlying those blocks, thus achieving infinite energy growth.
Since our system is not hyperbolic, the classical shadowing lemma for hyperbolic systems, saying that any pseudo-orbit can be `shadowed' by a true orbit, does not apply. We will show that, nevertheless, the  pseudo-orbits constructed in the previous section can be approximated by a true orbit. For this, we use a topological argument based on correctly aligned windows.
This argument is constructive and robust, so it allows us to also estimate the energy growth rate along the resulting orbit.

\subsection{Topological method}\label{section:topological}

In this section we review briefly the topological method of
correctly aligned windows, following
\cite{GideaZ04,GideaLlave06,GideaR07}. Earlier versions of the method go back to \cite{ConleyE1971,EastonMcGehee79,Easton81}.

A window is a triple consisting  of a mapping, a set, and a partition of the boundary of that set:
the mapping is a homeomorphism from a multi-dimensional rectangle in some Euclidean space to a manifold,
the set is the image of the multi-dimensional rectangle through the homeomorphism,
and  the partition divides the boundary of the set into an  exit set and an entry set, which play a dynamical role.

\begin{defn}
A $(n_1,n_2)$-window
in an $m$-dimensional manifold $M$, where $n_1+n_2=m$, is an ensemble $(W,W^{\rm exit}, W^{\rm entry},c )$ consisting of: \begin{itemize} \item[(1)] a homeomorphism
$c:\textrm{dom}(c )\to \textrm{im}(c)$, where
$\textrm{dom} (c )$ is an open neighborhood of
$[0,1]^{n_1}\times [0,1]^{n_2}\subseteq \mathbb{R}^m$, and
$\textrm{im}(c)$ is an open set in $M$,
\item[(2)] a
homeomorphic copy $W:=c\left ([0,1]^{n_1}\times [0,1]^{n_2}\right
)\subseteq \textrm{im}(c)$ of  $[0,1]^{n_1}\times [0,1]^{n_2}$ in~$M$,

\item[(3)] an `exit set'
\[W^{\rm exit}:=c\left(\partial[0,1]^{n_1}\times [0,1]^{n_2}
\right )\] and  an `entry set'  \[W^{\rm entry}:=c\left([0,1]^{n_1}\times
\partial[0,1]^{n_2}\right ).\]
\end{itemize}
\end{defn}
We adopt the following notation: $W_c=c^{-1}(W)=[0,1]^{n_1}\times [0,1]^{n_2}$,
$(W^{\rm exit} )_c=c^{-1}(W^{\rm exit} )=\partial[0,1]^{n_1}\times [0,1]^{n_2}$, and $(W^{\rm
entry} )_c=c^{-1}(W^{\rm entry})=[0,1]^{n_1}\times \partial[0,1]^{n_2}$. When the coordinate
system $c$ is evident from context, we suppress the subscript
$c$ from the notation.

Informally, two windows are correctly aligned under some map, provided that the
image of the first window under the map  crosses the second window
all the way through and across its exit set. Below we present a  version  of
the definition of correct alignment that is sufficient for the
purpose of this paper. More details can be found in \cite{GideaZ04}. Given two windows $(W_1,W_1^{\rm exit}, W_1^{\rm entry},c_1)$  and
$(W_2,W_2^{\rm exit}, W_2^{\rm entry},c_2)$ and a continuous map $f:M\to M$ with  $f(\textrm {im}(c_1))\subseteq
\textrm {im}(c_2)$, we will denote  $f_{c_1,c_2}=c_2^{-1}\circ f\circ c_1$.

\begin{defn}\label{defn:corr}
The window $W_1$ is correctly aligned with the window $W_2$ under $f$ if the following conditions
are satisfied:
\begin{itemize}
\item[(1)] In the case when $n_1,n_2\neq 0$, the conditions are:\begin{itemize}
\item[(1.i)]   There exists a continuous homotopy $h:[0,1]\times
(W_1)_{c_1} \to {\mathbb R}^{n_1} \times {\mathbb R}^{n_2}$ with
   \begin{eqnarray*}
      h_0&=&f_{c_1,c_2}, \\
      h([0,1],(W^{\rm exit}_1)_{c_1}) \cap (W_2)_{c_2} &=& \emptyset, \\
      h([0,1],(W_1)_{c_1}) \cap (W_2^{\rm entry})_{c_2} &=& \emptyset,\end{eqnarray*}
      \item[(1.ii)] There exists a linear map $A:{\mathbb R}^{n_1} \to
{\mathbb R}^{n_1}$, such that
\begin{itemize}
  \item[(1.ii.a)] $h_1(x,y)=(Ax,0)$ for $x \in [0,1]^{n_1}$ and $y \in [0,1]^{n_2}$,
    \item [(1.ii.b)]  $A(\partial[0,1]^{n_1}) \subset {\mathbb R}^{n_1}
    \setminus[0,1]^{n_1}$.
 \end{itemize}
   \end{itemize}
 \item[(2)] In the case when $n_2=0$, the conditions are:\begin{itemize}
\item[(2.i)]   $(W_1^{\rm exit})_{c_1}=(\partial W_1)_{c_1}$, $(W_2^{\rm exit})_{c_2}=(\partial W_2)_{c_2}$,
\item[(2.ii)] $(W_2)_{c_2}\subseteq \textrm{int}(f_{c_1,c_2}((W_1)_{c_1}))$,
   \end{itemize}
\item[(3)] In the case when $n_1=0$, the conditions are:\begin{itemize}
\item[(3.i)]   $(W_1^{\rm exit})_{c_1}=\emptyset$, $(W_2^{\rm exit})_{c_2}=\emptyset$,
\item[(3.ii)] $\textrm{int}(W_2)_{c_2}\supseteq f_{c_1,c_2}((W_1)_{c_1})$.
\end{itemize}\end{itemize}
\end{defn}

The correct alignment of windows is robust, in the sense that if two
windows are correctly aligned under a map, then they remain
correctly aligned under a sufficiently small $C^0$-perturbation of the
map. This allows to verify the correct alignment of long, finite sequences of windows,
by breaking them into shorter, finite sequences of windows whose correct alignment can be
easily controlled by perturbative arguments. This property will not be used in this paper.

Also, the correct alignment satisfies a natural product property.
Given two windows and a map, if each window can be written as a
product of window components, and if the components of the first
window are correctly aligned with the corresponding components of
the second window under the appropriate components of the map, then
the first window is correctly aligned with the second window under
the given map. We
refer to \cite{GideaLlave06,GideaR07} for details.

The following result can be thought of  as a topological version of
the Shadowing Lemma. Note that this result does not assume that the system is
hyperbolic.
\begin{thm}[\cite{GideaZ04}]\label{theorem:detorb}
%[Existence of orbits with prescribed trajectories]
Assume that $\{W_i\}_{i\in\mathbb{Z}}$ is a bi-infinite sequence of $(n_1,n_2)$-windows in $M$, and  $\{f_i\}_{i\in\mathbb{Z}}$ are continuous
maps on $M$. If $W_i$ is correctly aligned with $W_{i+1}$ under $f_i$ for every $i\in\mathbb{Z}$, then
there exists a point $p\in W_0$ such that
\[(f_{i}\circ \ldots\circ f_{0})(p)\in W_{i+1},\]
 for all $i\in\mathbb{Z}$.

Assume now that $\{W_i\}_{i\in\{0,\ldots,d\}}$ is a  finite sequence of $(n_1,n_2)$-windows in $M$, and $\{f_i\}_{i\in\{0,\ldots,d\}}$ are continuous
maps on $M$. If $W_i$ is correctly aligned with $W_{i+1}$ under $f_i$ for every $i=0,\ldots, d-1$, and $W_d$ is correctly aligned with $W_0$ under $f_d$, then
there exists a point $p\in W_0$ such that
\[(f_{d}\circ \ldots\circ f_{0})(p)=p.\]
\end{thm}

A sequence of windows $\{W_i\}_{i}$ as above will be referred at as a sequence of correctly aligned windows. Note that the verification of the correct alignment of the sequence amounts to verifying correct alignment relations between successive pairs $W_i$, $W_{i+1}$.
A consequence of this fact,  which is important for our applications,  is
that the concatenation of finite sequences of correctly aligned
windows is a finite sequence of correctly aligned windows. That is, if $W_0,\ldots, W_k$ forms a sequence of correctly aligned windows, and $W_k,\ldots, W_l$ forms a sequence of correctly aligned windows, then  $W_0,\ldots, W_l$ also forms a sequence of correctly aligned windows (for simplification, we omit here to specify the mappings under which the correct alignment is realized).

In the context of this paper, the maps $f_i$ from Theorem \ref{theorem:detorb}  will be  different powers of the time-$1$ map associated to the flow.

We emphasize again the difference between hyperbolic dynamics and correct alignment of windows.
To assert that an orbit is hyperbolic one needs to examine the expansion and contraction rates of the derivative of the map along the whole orbit. The concatenation of hyperbolic segments could fail to be hyperbolic (if the stable and unstable directions do not match). In contrast, to assert that a sequence of windows is correctly aligned, one only needs to verify that the image of one window under the map is correctly aligned with the next window in the sequence. Also, concatenations of finite sequences of correctly aligned windows are correctly aligned.

\subsection{Reduction to a discrete dynamical system}
\label{subsection:reduction}
We  reduce the perturbed geodesic flow
to  a discrete dynamical system by
considering the time-$1$ map of the flow  $\psi^\eps$ -- where the time refers to the
rescaled time $s$ in \eqref{eqn:hamrescaled} -- which we denote   $F_\eps$. Also, we denote by $\chi^\eps_1$ the time-$1$ map (relative to the rescaled time $s$) associated to the flow  $\chi^\eps$ on $N$.
Since $s=t/\eps$,
 the time-$1$ map relative to the rescaled  time $s$ is the time-$\eps$ map relative to the physical
time $t$.

The distinguished geometric objects for the perturbed geodesic flow, described in
the earlier sections, give rise to similar objects for
the discrete dynamical system.
We start by listing these objects and
summarizing their properties.
\begin{itemize}
\item[B1.] The map $F_\eps:T^*M\times N\to T^*M\times N$ is a
$C^{r -1}$-diffeomorphism.

\item[B2.] The manifold
$\tilde\Lambda_\eps=\Lambda_\eps\times N\subseteq T^*M\times N$  is a normally hyperbolic
invariant manifold for $F_\eps$, of dimension $(d+2)$; this manifold
has stable and unstable manifolds $W^s(\tilde\Lambda_\eps)$ and
$W^u(\tilde\Lambda_\eps)$, of dimension $(d+n+1)$.

\item [B3.] There exist exponential  rates
$0<\lambda_-<\lambda_+<\lambda_1<1<\mu_1<\mu_-<\mu_+$ such that
$\lambda_-<\|{DF_\eps}_{\mid E^s_{\eps,z}}\|<\lambda_+$,
$\mu_-<\|{DF_\eps}_{\mid E^u_{\eps,z}}\|<\mu_+$,
$\lambda_1<\|{DF_\eps}_{\mid E^c_{\eps,z}}\|<\mu_1$, where
$E^s_\eps$ and $E^u_\eps$ are the stable and unstable bundles in
the  decomposition $T_z(T^*M\times N)=T_z\tilde\Lambda_\eps
\oplus E_{\eps,z}^s\oplus E_{\eps,z}^u$. The above exponential rates
can be chosen independently of $\eps$.

\item [B4.] The stable and unstable manifolds
$W^s(\tilde\Lambda_\eps)$ and $W^u(\tilde\Lambda_\eps)$ have a
transverse intersection along a $(d+2)$-dimensional manifold
$\tilde\Gamma_\eps$.

\item [B5.] For the discrete dynamical system defined by $F_\eps$ there exists two  scattering maps $\tilde S^j_\eps:\tilde U_\eps^{j,-}\to \tilde U_\eps^{j,+}$ associated to two homoclinic channels $\tilde \Gamma^j_\eps$, $j=1,2$, satisfying Assumption (A4).   The set $\tilde U_\eps^{j,-}$ is of size $O(1)$ in the sense that there exist open sets $U^{j,-}\subseteq \Lambda_0$ such that $\tilde k_\eps(U^{j,-}\times N)\subseteq \tilde U_\eps^{j,-}$ for $j=1,2$, and for all $\eps$ sufficiently small, where $\tilde k_\eps:\tilde \Lambda_0=\Lambda_0\times N\to\tilde \Lambda_\eps$ is the parametrization of $\tilde \Lambda_\eps$ from Subsection \ref{subsection:scaled}.

\item[B6.] Consider the action-angle coordinates  $(J_\eps, \phi_\eps)$ on $\Lambda_\eps$. The restriction of $F_\eps $ to $\tilde\Lambda_\eps $
has the form
\[F_\eps(J_\eps, \phi_\eps, \theta)=(J_\eps+O(\eps^2), \phi_\eps+J_\eps+O(\eps^2), \chi^\eps_1(\theta)).\]

There exists $\tau>0$, independent of $\eps$, such that \[\frac{\partial (\pi_{\phi_\eps}\circ {F_\eps}_{\mid\Lambda_\eps})}{\partial J_\eps}(J_\eps, \phi_\eps) >\tau,\] for all
$(J_\eps, \phi_\eps, \theta)\in\tilde\Lambda_\eps$.
In particular ${F_\eps}_{\mid\Lambda_\eps}$ is an integrable twist map in the variables $(J_\eps, \phi_\eps)$, up
to order $O(\eps^2)$, with the twist coefficient lower bounded by $\tau$.

\item[B7.]  Each scattering map  $\tilde S^j_\eps:\tilde U_\eps^{j,-}\to \tilde U_\eps^{j,+}$,  associated to the  homoclinic channel $\tilde\Gamma^j$,  $j\in\{1,2\}$, is of the form
    \begin{eqnarray*}\tilde S^j_\eps(J_\eps^-, \phi_\eps^-, \theta^-)&=&(J_\eps^+, \phi_\eps^+, \theta^+),\quad\textrm{where}\\
 \phi_\eps^+&=& \phi_\eps^-+a+O(\eps^2),\\
 | J_\eps^+- J_\eps^-|&=&O(\eps^2),\\
 \theta^+&=& \theta^-.\end{eqnarray*}
 \item[B7.] There exists a sequence of elementary building blocks, of the type $B^j(J,\phi,\theta)$, $j=1,2$, as in  Subsection \ref{subsection:gainseq}, along which the scaled energy $H_\eps$ grows by $O(\eps)$ in a scaled time interval $\Delta s=O(1/\eps^2)$. Each building block consists of one application of one of the scattering maps, $\tilde S^j_\eps$, $j=1,2$, succeeded  by an orbit of the inner dynamics followed for a scaled time of $O(1)$. The succession  of  elementary building blocks is chosen as in Subsection \ref{subsection:gainseq}.
\end{itemize}

The general strategy to prove the Theorems \ref{thm:main1}, \ref{thm:main2}, \ref{thm:symbolic} is the following. We will choose a sequence of elementary building blocks as before, and will fix a  two-dynamics pseudo-orbit determined by it. We will prove that there exists a  true orbit that follows this pseudo-orbit.
Also, we will estimate the time it takes to such an orbit to perform the trip.

As explained in Subsection \ref{subsec:buildingblock}, an elementary building block consists of a segment of a homoclinic orbit $\psi^\eps_s(\tilde z_\eps)$, $s\in[T_-,T_+]$,  followed by a   trajectory segment $(\psi^\eps_s)_{\tilde\Lambda_\eps}(\tilde z_\eps^+)$ of the flow $\psi^{\eps}_s$ restricted to $\tilde \Lambda_\eps$, during a time interval of order $O(1)$.
Consider a sequence of $(n_1+1)$ successive elementary building blocks. For each one of them, we denote by $\tilde\gamma^\eps_{i}(s)$ the homoclinic segment of the elementary building block, and by
$\tilde\lambda^\eps_{i}(s)$ the trajectory segment of the flow $\psi^{\eps}_s$ restricted to $\tilde \Lambda_\eps$ corresponding to that building block, where $i=0,\ldots,n_1$. Also, we denote by $z^\eps_{i-1,i}$ the corresponding homoclinic point on $\tilde \gamma^\eps_i$, with $i\in\{1,\ldots,n_1\}$. We allow the considered homoclinic orbit segments to correspond to different homoclinic manifolds $\tilde\Gamma^j_\eps$, with $j=1,2$, but we do not make this distinction in the notation (for typographical reasons), since all estimates are uniform.

Each  trajectory segment that is a part of an elementary building block can be written as
$\tilde\gamma^\eps_{i}(s)=(\gamma^\eps_{i}(s),\theta(s))$, and $\tilde\lambda^\eps_{i}(s)=(\lambda^\eps_{i}(s),\theta(s))$, respectively.
Since a trajectory segment $\tilde\lambda^{\eps}_i(s)$ is followed for a time $O(1)$, the action $J_\eps$-coordinate along $\lambda^{\eps}_i(s)$  stays constant up to an $O(\eps^2)$ error.
Thus, to the  trajectory segment $\lambda^{\eps}_i(s)$  we can associate a level set $\{J_\eps=J_\eps^i\}$ of the $J_\eps$-coordinate in $\Lambda_\eps$, which is almost invariant up to $O(\eps^2)$.

For each $i\in\{0,\ldots,n_1\}$, denote $\mathcal{L}^\eps_i=\lambda^\eps_i\times N$; this is a $(1+d)$-dimensional manifold in $T^*M\times N$. The stable and unstable manifolds of $\mathcal{L}^\eps_{i}$ are $(n+d)$-dimensional manifolds
$W^s(\mathcal{L}^\eps_{i})=\bigcup_{z\in \mathcal{L}^\eps_{i}}W^s(z)$ and
$W^u(\mathcal{L}^\eps_{i})=\bigcup_{z\in \mathcal{L}^\eps_{i}}W^u(z)$, respectively, where the stable and unstable fibers of points are well defined due to the normal hyperbolicity of $\tilde\Lambda_\eps$. Note that these sets are not, in general, invariant, as their fibers are not invariant; we have
$F_\eps(W^s(z))\subseteq W^s(F_\eps(z))$ and $F_\eps^{-1}(W^u(z))\subseteq W^u(F_\eps^{-1}(z))$. By construction, we have that
for each $i\in\{1,\ldots,n_1\}$, $W^u(\mathcal{L}^\eps_{i-1})$ intersects transversally with
$W^s(\mathcal{L}^\eps_{i})$ along  $\{z^\eps_{i-1,i}\}\times N$.

To show the existence of an orbit that follows the pseudo-orbit determined by the given sequence of elementary building blocks,
we will use the stable and unstable   manifolds to build   a  chain of sets $\mathcal{L}^i_\eps$, with $i=0,\ldots,n_1$, around which we
will construct windows that are correctly aligned.
The  construction will be done according to the following sequence of steps, which we first describe informally below.

At the  first step, for each $i=1,\ldots,n_1$,
we construct a pair of windows $W^-_{i-1}$ and $W^+_{i}$ about
the heteroclinic intersection $(\{z^\eps_{{i-1},{i}}\}\times N)\cap\tilde\Gamma^j_\eps\subseteq  W^u(\mathcal{L}^\eps_{i-1})\cap W^s(\mathcal{L}^\eps_{i})$, such that
$W^-_{i-1}$ is correctly aligned with $W^+_{i}$ under the identity
mapping. Here the homoclinic intersection $\tilde\Gamma^j_\eps$ is one of the two homoclinic intersections defined by the distinguished homoclinic orbits $\gamma^j$,  $j\in\{1,2\}$, described in Assumption (A4).

At the second step we iterate $W^+_{i}$ forward in time
along $W^s(\mathcal{L}^\eps_{i})$  until its image is contained in some conveniently chosen neighborhood of  $\mathcal{L}^\eps_{i}$ in  $T^*M$, and we construct  a new window $\check W_{i}$ near $\mathcal{L}^\eps_{i}$, such that $W^+_{i}$ is correctly aligned with $\check W_{i}$ under some positive iterate of $F_\eps$.
Also, we iterate $W^-_{i-1}$ backwards in time along
$W^u(\mathcal{L}^\eps_{i-1})$   until its image is contained in some conveniently chosen neighborhood of
$\mathcal{L}^\eps_{i-1}$,  and we construct a new window
$\hat W_{i-1}$ about $\mathcal{L}^\eps_{i-1}$, such that
$\hat W_{i-1}$ is correctly aligned with $W^-_{i-1}$ under  some
positive iterate of $F_\eps$.
At this step, the neighborhoods of $\mathcal{L}^\eps_{i}$ and $\mathcal{L}^\eps_{i-1}$ are chosen so that the  system of coordinates from  Subsection \ref{subsection:coordinates}, near the normally hyperbolic invariant manifold $\tilde\Lambda_\eps$, is well defined in those neighborhoods. The $\Lambda_\eps$-component of the window $\hat W_{i-1}$ is  close to the level set $J_\eps=J_\eps^{i-1}$ of the action coordinate $J_\eps$, and the $\Lambda_\eps$-component of the  window $\check W_i$ is  close to the level set $J_\eps=J_\eps^{i}$ of the action coordinate $J_\eps$ on $\Lambda_\eps$.

At the third step we  align the window $\hat W_{i}$, corresponding to the heteroclinic
connection $W^u(\mathcal{L}^\eps_{i-1})\cap W^s(\mathcal{L}^\eps_{i})$,  to the window $\check W_{i}$, corresponding to the heteroclinic connection $W^u(\mathcal{L}^\eps_{i})\cap W^s(\mathcal{L}^\eps_{i+1})$. The $\Lambda_\eps$-components of these two windows are close to the same level set $J_\eps=J_\eps^i$ of the action coordinate. The construction at the third step concatenates the sequence of correctly aligned windows
constructed about one heteroclinic connection with the sequence of
correctly aligned windows constructed about the next heteroclinic
connection.

At the fourth step we concatenate short sequences of correctly aligned windows constructed as above, obtaining a long sequence of correctly aligned windows that follows the  pseudo-orbit underlying the sequence of elementary building blocks that achieves the desired energy growth.

The conclusion is that, once the windows have been constructed, the shadowing   result
Theorem \ref{theorem:detorb} will provide the existence of true
orbits that visit the windows in the prescribed order, hence these orbits will visit
  the prescribed level sets of the averaged action.

The construction of windows is  similar to that in
\cite{GideaLlave06}. Therefore we will describe most steps of the
construction succinctly. One step that is quite different is the
third step, in which we align two windows $\check W_i$ and $\hat
W_i$ about $\mathcal{L}^\eps_{i}$ under some iterate of the map
$F_\eps$. The difference is that in \cite{GideaLlave06} the map is a
twist map, while in our case, the map is a twist map in one
component and a time-$1$ map of some general flow in the other
component. We will explain this step of the construction in a  greater
detail.

Now we explain how the above strategy is used to prove the statements from Theorems  \ref{thm:main1}, \ref{thm:main2}, \ref{thm:symbolic}.

To prove that there exist trajectories along which the energy grows unboundedly we proceed as follow. We fix a potential $V\in\mathcal{V}'$, i.e., satisfying condition \textit{A4}.
Under the hypotheses of Theorem \ref{thm:main1} we choose an initial condition $(J_\eps^0,\phi_0,\theta_0)$, where $\theta_0$ is chosen to be a non-trivial uniformly recurrent point, $\phi_0$ is fixed as in condition \textit{A4}. Under the hypotheses of Theorem \ref{thm:main2} we choose the same initial condition with an arbitrary $\theta_0$.
Staring with this initial condition, we construct a finite sequence of elementary building blocks as in Subsection \ref{subsec:buildingblock}, with the  underlying pseudo-orbit
\[\tilde\gamma^\eps_0, \tilde\lambda^\eps_0,  \ldots, \tilde\gamma^\eps_{n_1}, \tilde\lambda^\eps_{n_1},\]
for some $n_1=O(1/\eps^3)$, along which the scaled energy grows from $E=1$ to $E=2$ in a scaled time $O(1/\eps^3)$.

This corresponds to a growth of the physical energy by $O(1/\eps^2)$ in a physical time $O(1/\eps^2)$. This physical energy growth is linear in the physical time.  If the initial energy level of the physical energy is $E_*$, the physical energy at the end of this sequence of elementary building blocks is $2E_*$.

The windowing construction from above above provides  a finite sequence of correctly aligned windows, whose $\Lambda_\eps$-components are close to these level sets. Any shadowing orbit obtained from Theorem \ref{theorem:detorb} will result in an growth of the scaled energy of $O(1)$ in a scaled time $O(1/\eps^3)$.

Note that, as remarked before, segments of correctly aligned windows can be concatenated. Since the constructions are uniform for all $\eps<\eps_0$, we can construct infinitely many segments that make the energy grow to $\infty$, and concatenate them.

More precisely, to obtain orbits whose scaled energy grows to infinity we proceed as in Subsection \ref{subsection:unbounded}. We reset  $\eps$ to $\eps=1/{\sqrt {2E_*}}$ and we repeat the construction of a sequence of elementary building blocks  whose scaled energy grows from $E=1$ to $E=2$ in a scaled time $O(1/\eps^3)$. The energy growth rate is still linear, and  at least as large as in the previous step. The physical energy grows from $2E_*$ to $4E_*$. Then we concatenate the segment of correctly aligned windows that grows the physical energy from $E_*$ to $2E_*$, with the segment that grows the physical energy from $2E_*$ to $4E_*$. The concatenation of these segments of correctly aligned windows constructed  is also a segment of correctly aligned windows. This construction of segments of elementary building blocks and corresponding segments of correctly aligned windows can be repeated indefinitely.
The  shadowing orbit that visits the resulting infinite sequence of correctly aligned windows yields an infinite physical energy growth at a linear rate with respect to the physical time.    In this way we obtain the statement  on the  existence of orbits with unbounded energy growth from   Theorems \ref{thm:main1} and \ref{thm:main2}.

To obtain an orbit whose scaled energy follows a prescribed energy path, we note that a path $\mathcal{E}:[0,\infty)\to\mathbb{R}$ determines, via time-discretization,  a sequence $J_\eps$-action level sets $\{\sqrt{2\mathcal{E}(k)}\}_{k\in\mathbb{N}}$ in $\Lambda_\eps$.  As in Subsection \ref{subsection:symbolic}, we construct a sequence of elementary building blocks whose corresponding $J_\eps$-values visit, in the prescribed order,   the values $\sqrt{2\mathcal{E}(k)}$, $k\in\mathbb{Z}$. Since, in general, one cannot move from one value $\sqrt{2\mathcal{E}(k)}$ to the next $\sqrt{2\mathcal{E}(k+1)}$ via a single elementary building block, the construction yields a sequence of action level sets $\{J^i_\eps\}_{i\in\mathbb{Z}}$, such that successive elementary building blocks correspond to successive value of $J^i_\eps$, and there exists a subsequence $\{i_k\}_{k\in\mathbb{Z}}$ of $\mathbb{N}$ such that $J^{i_k}_\eps=\sqrt{2\mathcal{E}(k)}$. Thus, we are in the same situation as described above, and we can proceed in the same way.

\subsection{Construction of windows}\label{section:constrwindows}

In this subsection we will work out the details of the construction of correctly aligned windows described earlier.

We consider a
finite sequence of elementary building blocks as in Subsection \ref{subsec:buildingblock}, with the  corresponding pseudo-orbit
\[\tilde\gamma^\eps_0, \tilde\lambda^\eps_0,  \ldots, \tilde\gamma^\eps_{n_1}, \tilde\lambda^\eps_{n_1},\]
for some $n_1=O(1/\eps^3)$. The corresponding action $J$ level sets corresponding to the curves  $\lambda^\eps_i$ in $\Lambda_\eps$, $i=0,\ldots,n_1$, are
\[J^{0}_\eps, J^{1}_\eps, \ldots, J^{n_1}_\eps.\]

\subsubsection{A system of coordinates near a normally hyperbolic invariant manifold}\label{subsection:coordinates}
The construction of windows requires a coordinate
systems relative to which the windows can be defined.

We now describe a general construction of a system of coordinates near a normally hyperbolic invariant manifold.
Given a normally hyperbolic invariant manifold $\Lambda\subseteq M$ for $F:M\to M$, the  bundle $N_x=E_x^u\oplus E_x^s$, $x\in\Lambda$,  gives a concrete realization of the normal bundle $N_{\Lambda}$. That is, $N_x$ is the complementary subspace to $T_x\Lambda$, i.e., $T_xM=T_x\Lambda\oplus N_x$, for all $x\in\Lambda$.

For any point $p$ in a sufficiently small neighborhood a neighborhood $\mathcal{N}(\Lambda)$ of $\Lambda$ in $M$, we can find  unique $x\in\Lambda$, $s\in E^s_x$, and $u\in E^u_x$, with $u,s$ small,  such that $p=\exp_x(s+u)$. Therefore, it is natural to use $(x,s,u)$ as a system of coordinates in $\mathcal{N}(\Lambda)$. We denote $h(x,s,u)=\exp_x(s+u)$, where $\exp:TM\to M$ is the exponential map.

In general, this system of coordinates is as smooth as the bundles $E^s,E^u$, whose regularity is limited by the regularity of the map $F$ and by the ratios of the exponential rates.

In our case, because the exponential rate of $F_\eps$ on the manifold $\tilde\Lambda_\eps$ is close to $0$ for $\eps$ small enough, we have that the corresponding maps $h_\eps$, $h^{-1}_\eps$ are $C^r$.  We also note that if we express $F_\eps$ in terms of the coordinate mapping, we have, for $F_\eps\in C^2$,
\[h_\eps^{-1}\circ F_\eps\circ h_\eps(x,s,u)=(F_\eps(x),DF_\eps(x)s,DF_\eps(x)u)+O(\delta^2),\]
where $\delta>0$ is the size of the neighborhood $\mathcal{N}(\Lambda_\eps)$. Furthermore
\[Dh^{-1}_\eps\circ F_\eps\circ h_\eps^{-1}(x,s,u)=DF_\eps(x)+O(\delta).\]

If we express $F_\eps$ relative this system of coordinates, $\tilde F_\eps=h_\eps^{-1}\circ F_\eps\circ h_\eps$ is close to the tangent map
\[TF_\eps(x,s,u)=(F_\eps(x), DF_\eps(x)s,DF_\eps(x)u).\]

Indeed, we have
\[\begin{split}
\|\tilde F_\eps-TF_\eps\|_{C^0}\leq C\delta^2,\\
\|\tilde F_\eps-TF_\eps\|_{C^1}\leq C\delta,
\end{split}
\]
for some $C>0$, where $\delta$ is the size of the neighborhood $\mathcal{N}(\Lambda_\eps)$. The constants $\delta$ and $C$ can be chosen independent of $\eps$.

In particular, by considering a sufficiently small neighborhood  $\mathcal{N}(\Lambda_\eps)$ we can ensure that the map $F_\eps$ is contracting the $s$-components, expanding in the $u$-components, and more or less neutral in $x$.

Noting that we can describe the manifold $\tilde\Lambda_\eps$ via the coordinates $(J,\phi,\theta)$, the above construction provides us with a $C^1$-smooth coordinate
system $(J,\phi, \theta,  s, u)$ in a neighborhood $\mathcal{N}(\tilde \Lambda_\eps)$ of $\tilde \Lambda_\eps$. Relative to  this coordinate system
the map $F_\eps$ can be approximated by a skew product of a map acting in the center
directions of $\tilde\Lambda_\eps$ and a map acting in the
hyperbolic directions.  To simplify notation,  we will not make explicit the dependence on $\eps$ of the coordinate systems. Also, from this point on we will denote $F_\eps$ by $F$.

We will use this system of coordinates to construct windows near $\Lambda_\eps$, and, in particular, about the sets $\mathcal{L}^\eps_i$.

We will also need to  construct windows about the heteroclinic intersections
$(\{z^\eps_{i-1,i}\}\times N)\cap \tilde\Gamma_\eps$  of
$W^u(\mathcal{L}^\eps_{i-1})$ with $W^s(\mathcal{L}^\eps_{i})$.
For this, we propagate the above  coordinate system   from a
neighborhood of the set $\mathcal{L}^\eps_{i-1}$ in
$\tilde\Lambda_\eps$ along the unstable manifold
$W^u(\mathcal{L}^\eps_{i-1})$, giving rise to a
$C^1$-smooth coordinate system $(J^-,\phi^-, \theta,  s^-, u^-)$ in a neighborhood of
$(\{z^\eps_{i-1,i}\}\times N)\cap \tilde\Gamma_\eps$.
Also, we propagate the same coordinate system from a
neighborhood of the leaf $\mathcal{L}^\eps_{i}$ along the
stable manifold $W^s(\mathcal{L}^\eps_{i})$ to a neighborhood
of the heteroclinic intersection $(\{z^\eps_{i-1,i}\}\times N)\cap \tilde\Gamma_\eps$, producing a $C^1$-smooth
coordinate system $(J^+,\phi^+, \theta,  s^+, u^+)$ in the neighborhood of
$(\{z^\eps_{i-1,i}\}\times N)\cap \tilde\Gamma_\eps$.

The two coordinate systems differ one from the
other by order $O(1)$, that is, if $\Phi$ denotes the coordinate
change from one system to another then $C_3^{-1}\leq \|D\Phi\|\leq
C_3$ uniformly in a compact neighborhood of $\tilde \Gamma_\eps$,
for some $C_3>1$. Since we obtain two coordinate systems around $(\{z_{i-1,i}\}\times N)\cap \tilde\Gamma_\eps$, we can construct windows and verify their correct alignment in either coordinate system. The sizes of the components of a window in one coordinate system will differ by the sizes in the other coordinate system by some multiplicative constants that are independent of $\eps$ and, due to compactness,  they can be chosen to be the same for all heteroclinic intersections.

\subsubsection{Choice of constants}\label{subsection:constants}
We choose some constants that will be used throughout the proof.
Let $a^\pm$ be defined as follows
\[a^\pm=\sup_{z\in H^{-1}[1,2]\cap\tilde\Gamma_\eps}d(z^\pm, z), \]
where the distance is measured along the stable (unstable) fiber through $x$, respectively.
Since $H^{-1}[1,2]\cap\tilde\Gamma_\eps,H^{-1}[1,2]\cap\tilde\Lambda_\eps$ are compact, then $0<a^\pm<\infty$.

First we choose some constant $\eps_1>0$, independent of $\eps$, and then we
choose the positive constants
$\alpha^-,\alpha^+,\check\alpha,\hat\alpha$ and
$\beta^-,\beta^+,\check\beta,\hat\beta$, independent of $\eps$, such
that
\begin{eqnarray}
\label{eqn:const1} 3\eps_1<\alpha^-=\hat\alpha=\check\alpha
<C^{-1}_3\alpha^+,\\
\label{eqn:const2} 3\eps_1<\beta^+=\hat\beta=\check\beta
<C^{-1}_3\beta^-.
\end{eqnarray}

Second,  we choose positive integers $N,M$ sufficiently large
so that
\begin{eqnarray}
\label{eqn:const3}\lambda^N_+(a^++\alpha_+)<2\eps_1,\, \check\beta+\eps_1<\mu_-^N\beta^+,\, \mu^{-M}_-(\beta^-+a^-)<2\eps_1, \\
\label{eqn:const4} F^N(\tilde\Gamma_\eps)\subseteq \mathcal{N}(\tilde\Lambda_\eps),\, F^{-M}(\tilde\Gamma_\eps)\subseteq \mathcal{N}(\tilde\Lambda_\eps),
\end{eqnarray}
where $\mathcal{N}(\tilde\Lambda_\eps)$ is the neighborhood of $\tilde\Lambda_\eps $  where the coordinate system described in Subsection \ref{subsection:coordinates} is defined. Note that there exist finite $N,M$ as in \eqref{eqn:const4} due to the definition of $\tilde\Gamma_\eps$. Since   the dynamical system \eqref{eqn:hamrescaled} tends to a product system when $\eps\to 0$, then  we can choose  $N,M$ to be independent of $\eps$ for all  $\eps$  sufficiently small.

Third, we choose a positive integer $K$, satisfying the following condition.
\begin{equation}
\label{eqn:const5}2\check\gamma+\hat\gamma+3\eps_1 <K\tau\check\delta,
\end{equation}
where $\tau$ is the twist constant from (B6).
Such a finite $K$ exists, for all small enough $\eps$, since  the inner map $F$ is a twist map in the $(J, \phi)$-variables.

Fourth, we choose the positive constants
$\delta^-,\delta^+,\check\delta,\hat\delta,\gamma^-,\gamma^+,\check\gamma,\hat\gamma$,
such that they satisfy the following inequalities
\begin{eqnarray}
\label{eqn:const6}\check\delta<\check\delta+\eps_1<\delta^+<C^{-1}_3\gamma^-<\gamma^-
<\gamma^-+M\tau\delta^-+\eps_1<\hat\gamma,\\
\label{eqn:const7}\check\delta<\check\delta+\eps_1<\hat\delta<\hat\delta+\eps_1<\delta^-<C_3\delta^-
<\gamma^+<\gamma^++N\tau\delta^++\eps_1<\check \gamma.
\end{eqnarray}
Additionally, we want these constants to be sufficiently small. We will explain later in the argument how small should they be, but we emphasize here that the smallness condition is independent of $\eps$ and can be made precise from the beginning of the argument.

Fifth, we make $\eps$ even smaller, if necessary, as described below. Note that the estimates on the inner map ${F}_{\mid \tilde\Lambda_\eps}$ and on the outer map $S_\eps$ involve some error terms of orders $O(\eps^2)$, as in B6 and B7. Due to the compactness of $H^{-1}[1,2]\cap\tilde \Lambda_\eps$ and $H^{-1}[1,2]\cap\tilde\Gamma_\eps$, these error terms can be bounded from
above  by  $C_4\eps^2$, for some   constant $C_4>0$ independent of $\eps$. Now we choose $\eps$ sufficiently small so that
\begin{equation}\label{eqn:const8}
NC_4\eps^2<\eps_1,\, MC_4\eps^2<\eps_1, KC_4\eps^2<\eps_1.
\end{equation}
Thus, when we will estimate the error terms when iterating the inner map or the outer map up to $\max\{M,N,K\}$ times, we will be able to conclude that the error terms are always less than $\eps_1$.

From now on $\eps$ is sufficiently small and fixed.

\subsubsection{Step 1.} Let us consider the heteroclinic
intersection $W^u(\mathcal{L}^\eps_{i-1})$ with
$W^s(\mathcal{L}^\eps_{i})$ at $\{z^\eps_{i-1,i}\}\times N$. As seen before,
the normal hyperbolicity implies that there exist $z_{i-1,i}^{\eps,-}\in
\mathcal{L}^\eps_{i-1}$ and $z_{i-1,i}^{\eps,+}\in \mathcal{L}^\eps_{i}$
such that $z^\eps_{i-1,i}\in W^u(z_{i-1,i}^{\eps,-})\cap W^s(z_{i-1,i}^{\eps,+})$. Let
$a^-_{i-1}$ be the distance between $z^\eps_{i-1, i}$ and $z_{i-1,i}^{\eps,-}$
measured along $W^u(\mathcal{L}^\eps_{i-1})$,  and let
$a_{i}^+$ be the distance between $z^\eps_{i-1, i}$ and $z_{i-1,i}^{\eps,+}$
measured along $W^s(\mathcal{L}^\eps_{i})$.  We have that
$a^+_i<a^+$ and $a^-_i<a^-$.

At this step we construct a pair of windows $W^-_{i-1}, W^+_{i}$
about $(\{z^\eps_{i-1, i}\}\times N)\cap\tilde\Gamma_\eps$, such that $W^-_{i-1}$ is correctly
aligned with $W^+_{i}$ under the identity mapping.

We define the window $W^-_{i-1}$ in the coordinates
 $(s^-, u^-,\phi^-,J^-,\theta)$  of the type
\[W^-_{i-1}= (R^{
s^-}_{i-1}\times R^{u^-}_{i-1})\times
(R^{\phi^-}_{i-1}\times R^{J^-}_{i-1})\times
R^{\theta}_{i-1} ,\] where $R^c$ denotes a rectangle in
the coordinate $c$. We will think of $W^-_{i-1}$ as a product of
three window  components: $R^{s^-}_{i-1}\times R^{u^-}_{i-1}$,
corresponding to the hyperbolic coordinates, $R^{\phi^-}_{i-1}\times R^{J^-}_{i-1}$,
corresponding to the angle-action coordinates, and $R^{
\theta}_{i-1} $ corresponding to the external system on $N$.

We define the exit set $(W^-)^{\rm exit}_{i-1}$ of $W^-_{i-1}$ by
\begin{eqnarray*}(W^-)^{\rm exit}_{i-1}=&&(R^{
s^-}_{i-1}\times \partial R^{u^-}_{i-1})\times
(R^{\phi^-}_{i-1}\times R^{J^-}_{i-1})\times
R^{\theta}_{i-1}  \\
& \cup& (R^{
s^-}_{i-1}\times R^{u^-}_{i-1})\times
(\partial R^{\phi^-}_{i-1}\times R^{J^-}_{i-1})\times
R^{\theta}_{i-1}  \\ & \cup&(R^{
s^-}_{i-1}\times R^{u^-}_{i-1})\times
(R^{\phi^-}_{i-1}\times R^{J^-}_{i-1})\times
\partial R^{\theta}_{i-1}.
\end{eqnarray*}
This means that the exit directions of $W^-_{i-1}$ correspond to the
unstable direction  $u^-$ of the hyperbolic window
$R^{s^-}_{i-1}\times R^{u^-}_{i-1}$, to the
angle direction  $\phi^-$ of the action-angle window
$R^{\phi^-}_{i-1}\times R^{J^-}_{i-1}$, and to all directions of  the external-state window $R^{\theta}_{i-1}$. This definition of the exit set follows
the construction of product of windows described in \cite{GideaLlave06}.

Similarly, we define the window $W^+_{i}$ in the coordinates $(
s^+, u^+,  \phi^+, J^+,
\theta)$
 to be given by the product
\[W^+_{i}=(R^{
s^+}_{i}\times R^{u^+}_{i})\times
(R^{\phi^+}_{i}\times R^{J^+}_{i})\times
R^{\theta}_{i} ,\] and its exit set $(W)^{\rm exit}_{i}$ by
\begin{eqnarray*}
(W^+)^{\rm exit}_{i}=&&(R^{
s^+}_{i}\times \partial R^{u^+}_{i})\times
(R^{\phi^+}_{i}\times R^{J^+}_{i})\times
R^{\theta}_{i} \\&\cup&(R^{
s^+}_{i}\times R^{u^+}_{i})\times
(R^{\phi^+}_{i}\times \partial R^{J^+}_{i})\times
R^{\theta}_{i} \\ & \cup& (R^{
s^+}_{i}\times R^{u^+}_{i})\times
(R^{\phi^+}_{i}\times R^{J^+}_{i})\times
\partial R^{\theta}_{i}.
\end{eqnarray*}

The exit directions of $W^+_{i}$ correspond to the
direction $u^+$ of $R^{s^+}_{i}\times R^{
u^+}_{i}$, to the direction    $J^+$ of
$R^{\phi^+}_{i}\times R^{J^+}_{i}$, and to all
directions of $R^{\theta}_{i}$. Note that the only
difference in the exit directions of $W^+_{i}$ from
$W^-_{i-1}$ is the switching from the direction $\phi^-$ to
the direction $J^+$ in the angle-action components. Since we are
not using yet any dynamics  in constructing these windows, this
switching in the exit direction of the action-angle
components may seem arbitrary. The reason for this switching  will
become apparent later when we align windows by the inner map:  when an action-angle component is iterated under the inner dynamics until it crosses another action-angle   component, the twist property of the inner map causes  the first window to have its action direction stretched across  the second window along its angle direction.  Therefore we will
have a switching of the exit directions due to the alignment under
the inner map. By having another switching of the exit direction at
the heteroclinic intersection, we ensure the consistency of the correct
alignment in the two-step process. This will be explained in greater
detail at  Step 3.

We set the sizes of the rectangular components of the windows as follows:
\begin{eqnarray}\label{eqn:order1}
\|R^{s^-}_{i-1}\|=\alpha^-,\, \|R^{u^-}_{i-1}\|=\beta^-\\\nonumber\|R^{\phi^-}_{i-1}\|=\gamma^-,\,\|R^{
J^-}_{i-1}\|=\delta^-.
\end{eqnarray}
\begin{eqnarray}\label{eqn:order1b}
\|R^{s^+}_{i}\|=\alpha^+,\,\|R^{u^+}_{i}\|=\beta^+,\\\nonumber
\|R^{\phi^+}_{i}\|=\gamma^+,\,\|R^{J^+}_{i}\|=\delta^+.
\end{eqnarray}
Above, by $\|R^c\|$ we mean the diameter of the rectangle $R^c$ of the projection onto the coordinate~$c$.

We want to ensure that $W^-_{i-1}$ is correctly aligned with $W^+_i$
under the identity mapping. The two windows are defined in two
different coordinate systems, so we need to use the coordinate
change $\Phi$ to compare them relative to the same coordinate
system. We therefore require that $\Phi(R^{
s^-}_{i-1})\subseteq \textrm{int} (R^{s^+}_{i})$,
$\textrm{int} (\Phi(R^{u^-}_{i-1}))\supseteq R^{u^+}_{i}$,
$\textrm{int} (\Phi(R^{\phi^-}_{i-1}))\supseteq R^{J^+}_{i}$,
$\Phi(R^{J^-}_{i-1})\subseteq \textrm{int} (R^{
\phi^+}_{i})$, and $R^{\theta }_{i}\subseteq
\textrm{int}(\Phi(R^{\theta }_{i-1}))$. These conditions
imply that $W^-_{i-1}$ is correctly aligned with $W_i^+$ under the
identity mapping. These conditions are consistent with the size
specifications in \eqref{eqn:order1} and \eqref{eqn:order1b}, due
to the choice of constants  in \eqref{eqn:const1},
\eqref{eqn:const2}, \eqref{eqn:const3}, and \eqref{eqn:const4}.

\subsubsection{Step 2.} At this step we take a  forward iterate  of the window $W^+_{i}$
along the stable manifold $W^s(\mathcal{L}^\eps_{i})$, and  we align its image with a window
$\check W_{i}$ near $\tilde\Lambda_\eps$. Similarly, we take a  backwards iterate of the window   $W^-_{i}$ along the unstable manifold $W^u(\mathcal{L}^\eps_{i-1})$, and we align its image with a window $\hat W_{i}$
near $\tilde\Lambda_\eps$.

First, we consider the positive iterate $F^{N}(W_i^+)$ of $W_i^+$, with $N$ as in Subsection \ref{subsection:constants}.  Since the window $W^+_i$ has been defined in the coordinate system $( s^+,  u^+,  \phi^+, J^+, \theta)$ near $\tilde\Gamma_\eps$, which was obtained by propagating the  system near $\tilde \Lambda_\eps$ along the stable manifold,  and  $N$  as in \eqref{eqn:const4}, the image  $F^{N}(W_i^+)$  is naturally defined as a multi-dimensional rectangle in  the coordinate system $( s,  u,
\phi, J, \theta)$ defined in  $\mathcal{N}(\mathcal{L}^\eps_{i})$. The dynamics in these coordinates is the skew product of the dynamics in the center directions with the dynamics in the hyperbolic directions.

The twist condition  satisfied by $F$ in the $(\phi,J)$ variables, described in (B6), determines a sheering
of the $\phi$-direction by a quantity of $\tau J$ per iterate along each level set of the action variable $J$,
modulo an error  up to $\eps_1$; also, the $J$-coordinate is preserved up to  $\eps_1$. Since $F^{N}(W^+_i)$ is in $\mathcal{N}(\tilde\Lambda_\eps)$,  there exists a rectangle $\check R^{\phi}_{i}\times
\check R^{ J^+}_{i}$ in the  $( \phi, J)$ coordinate such that the $( \phi, J)$ component of $W^+_i$ is correctly aligned under
$F^{N}$ with  $\check R^{\phi}_{i}\times\check R^{ J^+}_{i}$. Since the size of the $( \phi, J)$  component of $W^+_i$  is  $\gamma^-\times\delta^-$, we can choose  $\check R^{\phi}_{i}\times
\check R^{ J^+}_{i}$ to be of size $\|\check
R^{\phi}_{i}\|=\check\gamma>\gamma^++N\tau\delta^++\eps_1$, and
$\|\check R^{ J^+}_{i}\|=\check\delta<\delta^+-\eps_1$. These inequalities are justified by \eqref{eqn:const6} and \eqref{eqn:const7}.

In the $\theta$-direction the map $F$ acts as the time-$\eps$ map $\chi$ of the flow $\chi$.
To achieve correct alignment of the windows in the $\theta$-variable,  it is sufficient to choose $\check R^{
\theta}_{i}$ a topological rectangle in the interior of $\chi^{N}(R^{
\theta+}_{i})$.

%(In principle, we can choose $\check R^{\bar \theta}_{i}=F^{N_i}(R^{\bar \theta+}_{i})$.)

The distance between the point $F^{N}(z^\eps_{i-1,i})\in
F^{N}(W^+_i)$ and $F^{N}(z^{\eps,+}_{i-1,i})$, measured along the
stable manifold $W^s(\mathcal{L}^\eps_{i})$, is less than
$\lambda_+^{N}a^+$.

The size of the projection of $F^{N}(W^+_i)$ onto the $s$-coordinate  is
 at most $\lambda_+^{N} \alpha^++\eps_1$. The size of the $u$-component of
$F^{N}(W^+_i)$ is at least $\mu_- ^{N}\beta^+-\eps_1$. By the choice of the coordinates in Subsection \ref{subsection:coordinates},  the
hyperbolic directions of $F^N(R^{s^+}_{i}\times R^{u^+}_{i})$
coincide with the hyperbolic directions in $\mathcal{N}(\tilde\Lambda_\eps)$.
Since by \eqref{eqn:const1} and \eqref{eqn:const3} we have
$\check\alpha>\lambda^{N}(\alpha^++a^+)+\eps_1$, and by
\eqref{eqn:const2} and \eqref{eqn:const4}, we have $\check
\beta+\eps_1<\mu_-^{N }\beta^+$, then we can construct a rectangle
$\check R^{ s}_{i}\times \check R^{ u}_{i}$ in the hyperbolic variables, of
size $\check\alpha\times\check\beta$,
such that $\check R^{ s^+}_{i}\times \check R^{
u^+}_{i}$ is correctly aligned with $\check R^{
s}_{i}\times \check R^{u}_{i}$ under $F^{N }$, relative to the
the hyperbolic variables.

Through these choices,  we construct a new window $\check W_i$ about
$z^{\eps,+}_{i-1,i}$, given by \[\check W_{i}=(\check R^{
s}_{i}\times \check R^{u}_{i})\times(\check
R^{\phi}_{i}\times \check R^{ J}_{i})\times \check
R^{ \theta}_{i},\] in the coordinates
 $( s, u,  \phi, J,
\theta)$, such that $F^{N}(W_i^+)$ is correctly aligned
with $\check W_i$ under the identity mapping, or, equivalently,
$W_i^+$ is correctly aligned with $\check W_i$ under $F^{N}$.
Its exit set $(\check W)^{\rm exit}_{i}$ is given by
\begin{eqnarray*}(\check W)^{\rm exit}_{i}=&&(\check R^{
s}_{i}\times\partial \check R^{ u}_{i})\times (\check
R^{\phi}_{i}\times \check R^{ J}_{i})\times \check
R^{ \theta}_{i}
\\ &\cup& (\check R^{
s}_{i}\times \check R^{u}_{i})\times (\check
R^{\phi}_{i}\times \partial\check R^{J}_{i})\times
\check R^{\theta}_{i}\\ &\cup& (\check R^{s}_{i}\times\check R^{u}_{i})\times (\check
R^{\phi}_{i}\times \check R^{J}_{i})\times
\partial \check R^{\theta}_{i}.
\end{eqnarray*}

We note that the correct alignment of $W_i^+$ with  $\check W_i$  under $F^{N}$ follows from the correct alignment of the window components, according to Definition \ref{defn:corr},  and the product property from \cite{GideaLlave06}.
The time it takes to achieve this alignment is $N$, which
is independent of $\eps$, and so is of order $O(1)$.

In a similar fashion, we construct a new window  $\hat
W_{i-1}$ near $\mathcal{L}^\eps_{ J_{i-1}}$ such that $F^{M}(\hat W_{i-1})$
is correctly aligned with $W^-_{i-1}$ under the identity mapping,
or, equivalently, $\hat W_{i-1}$ is correctly aligned with
$W^-_{i-1}$ under $F^{M}$, for some ${M}$ as in Subsection \ref{subsection:constants}. For this purpose, we take a negative iterate $F^{-M}(W^-_{i-1})$  of $W^-_{i-1}$   and we construct a new window
$\hat W_{i-1}$ about $z^{\eps,-}_{i-1,i}$, of the following type

\[\hat W_{i-1}=\hat R^{s}_{i-1}\times \hat R^{
u}_{i-1}\times\hat R^{\phi}_{i-1}\times \hat R^{
J}_{i-1}\times \hat R^{ \theta}_{i-1},\] in the
coordinates
 $( s, u,  \phi, J,
\theta)$,   with the exit set given by \begin{eqnarray*}(\hat W)^{\rm
exit}_{i-1}=&&\hat R^{s}_{i-1}\times\partial \hat R^{
u}_{i-1}\times \hat R^{\phi}_{i-1}\times \hat R^{
J}_{i-1}\times \hat R^{ \theta}_{i-1}
\\ &\cup& \hat R^{s}_{i-1}\times \hat R^{u}_{i-1}\times \partial\hat
R^{\phi}_{i-1}\times \hat R^{J}_{i-1}\times \hat
R^{\theta}_{i-1}\\ &\cup& \hat R^{
s}_{i-1}\times\hat R^{ u}_{i-1}\times \hat
R^{\phi}_{i-1}\times \hat R^{ J}_{i-1}\times
\partial \hat R^{ \theta}_{i-1}.
\end{eqnarray*}

We choose the size of the window components to be $\|\hat R_{i-1}^{ s}\|=\hat\alpha$, $\|\hat R_{i-1}^{ u}\|=\hat\beta$, $\|\hat R_{i-1}^{ \phi}\|=\hat\gamma$, and $\|\hat R_{i-1}^{ J}\|=\hat\delta$.
By the choice of the coordinates in Subsection \ref{subsection:coordinates}, we can choose    $\hat
R^{ s}_{i-1}\times \hat R^{ u}_{i-1}$ so that it is correctly
aligned with $R^{ s -}_{i-1}\times R^{ u -}_{i-1}$
under $F^{M}$. By  condition \textit{B6}  we can choose
$\hat
R^{ \phi}_{i-1}\times \hat R^{ J}_{i-1}$ so that it is correctly
aligned with $R^{ \phi -}_{i-1}\times R^{ J -}_{i-1}$
under $F^{M}$. In the $\theta$-variable  we choose $\hat R^{\theta}_{i-1}$ so that
$R^{\theta -}_{i-1}\subseteq \textrm{int}(\chi^{M}(\hat
R^{\theta}_{i-1}))$.

We also require that $(\hat R^{\phi}_{i-1}\times
\hat R^{ J}_{i-1})\times \hat R^{\theta}_{i-1}$ is contained, via the parametrization $k$ described in Subsection \ref{subsection:scaled},  in $U^-\times N\subseteq k^{-1}(\tilde U^-)$.

Now we verify that the choice of the sizes of the window components is compatible with these correct alignment relations.
Using \eqref{eqn:const1}, \eqref{eqn:const2}, \eqref{eqn:const4}, we can
ensure that $\hat\beta>\mu_-^{-M}(\beta^-+a^-)+\eps_1$ and
$\hat\alpha+\eps_1<\lambda^{-M}_+\alpha^-$. By \eqref{eqn:const6} and \eqref{eqn:const7} we can
ensure that $\hat \delta+\eps_1<\delta^-$ and
$\hat\gamma>\gamma^-+{M}\tau\delta^-+\eps_1$, where the term
$M\tau\delta^-$ represents  the effect of the twist map on $R^{
\phi -}_{i-1}\times R^{ J -}_{i-1}$ under $M$ negative
iterates.

By the product property of correctly aligned windows,  we obtain that $F^{M}(\hat W_{i-1})$ is correctly aligned
with $ W_{i-1}^-$ under the identity mapping, or, equivalently,
$\hat W_{i-1}$ is correctly aligned with $W_i^-$ under $F^{M}$.

We note that to achieve the correct alignment of the windows in Step 2 we do not use the Lambda Lemma as in \cite{Marco96,FontichM00,GideaLlave06}. The role of the Lambda Lemma in this paper is taken by the choice of  coordinates   coordinates near $\tilde\Lambda_\eps$ from Subsection \ref{subsection:coordinates}, that are extended along the stable and unstable manifolds to produce convenient coordinate systems about $\tilde\Gamma_\eps$.

\subsection{Step 3.} We consider the heteroclinic intersection
$W^u(\mathcal{L}^\eps_{i-1})\cap W^s(\mathcal{L}^\eps_{i})$, and the heteroclinic intersection
$W^u(\mathcal{L}^\eps_{i})\cap W^s(\mathcal{L}^\eps_{i+1})$.

By applying the  Step 2 for each heteroclinic connection  we obtain a pair of windows
$\hat W_i$ and $\check W_i$, both located about the set $\mathcal{L}^\eps_{i}$.

At this step we want to get $\check W_{i}$ correctly aligned with
$\hat W_{i}$ under some iterate $F^{K}$. We pointed out earlier
that the dynamics in  the coordinates $(\phi, J,\theta,  u, s)$ is the skew
product of the dynamics in the center directions and the dynamics in
the hyperbolic directions.

The first task is to ensure the correct alignment of the rectangle
$\check R^{ u}_{i}\times \check R^{ s}_{i}$, of
exit set $\partial\check R^{ u}_{i}\times \check R^{
s}_{i}$, with the  rectangle $\hat R^{ u}_{i}\times\hat R^{ s}_{i}$,
of exit set $\partial\hat R^{u}_{i}\times\hat R^{ s}_{i} $.
These rectangles are defined in the same coordinate system on $\mathcal{N}(\tilde\Lambda_\eps)$.
The correct alignment reduces to ensuring the following inclusions $\textrm{int}(F^{K}(\check
R^{ u}_{i}))\supseteq \hat R^{ u}_{i}$ and
$F^{K_i}(\check R^{ s}_{i})\subseteq \textrm{int}(\hat
R^{ s}_{i})$. Since the unstable directions get uniformly
expanded and the stable directions get uniformly contracted, it is
sufficient for the correct alignment to have $\check R^{
u}_{i}\times \check R^{ s}_{i}$ and  $\hat R^{
u}_{i}\times \hat R^{ s}_{i}$ of the same size
$\check\alpha\times \check\beta=\hat\alpha\times \hat\beta$; see
\eqref{eqn:const1} and \eqref{eqn:const2}.

The second task is to correctly align  $(\check R^{\phi}_{i}\times \check R^{ J}_{i})\times \check R^{ \theta}_{i}$ with
$(\hat R^{\phi}_{i}\times \hat R^{ J}_{i}) \times \hat R^{ \theta}_{i}$.
The exit set of $\check R^{\phi}_{i}\times \check R^{ J}_{i}$ is given
by $\check R^{\phi}_{i}\times \partial\check R^{ J}_{i}$, while the exit set  of $\hat R^{\phi}_{i}\times \hat
R^{ J}_{i}$ is given  by $\partial\hat
R^{\phi}_{i}\times \hat R^{ J}_{i}$.
Under   $F^{K}$, the rectangle $\check
R^{\phi}_{i}\times \check R^{ J}_{i}$ is transformed
into a topological rectangle whose exit set  components -- `top'
edge and `bottom' edge -- are shifted apart one from the other by at least $K \tau\check\delta-\eps_1$ in the
$\phi$-direction, and they are separated by a distance of
at least $\check\delta-\eps_1$ in the $
J$-direction. In order for $F^{K}(\check
R^{\phi}_{i}\times \check R^{ J}_{i})$ to be
correctly aligned with $\hat R^{\phi}_{i}\times \hat
R^{ J}_{i}$ under the identity, the image
$F^{K}(\check R^{\phi}_{i}\times \check R^{
J}_{i})$ should stretch across $\hat
R^{\phi}_{i}\times \hat R^{ J}_{i}$ in the
direction of $ \phi$ and its exit set components should
come out through the exit set components -- left edge and right edge
-- of $\hat R^{\phi}_{i}\times \hat R^{ J}_{i}$. To ensure the stretching all the way of the first window
across the second, we need that $( K\tau\|\check R^{
J}\| -\eps_1) - 2(\|\check
R^{\phi}_i\|+\eps_1)$, representing the shearing in the
$ \phi$ direction  of $F^{K}(\check R^{\phi}_{i}\times \check R^{ J}_{i})$  minus the size of the top and the bottom edge,   should be
bigger than $ \|\hat R^{\phi}_i\|$, representing     the `width'
of $\hat R^{\phi}_{i}\times \hat R^{ J}_{i}$ in the $ \phi$-direction.
Also, to ensure that the image  of the first rectangle does not meet
the entry set of the second rectangle, we need that the quantity
$\|\check R^{ J}\|+\eps_1$, representing an upper bound of the
`height' of $F^{K}(\check R^{\phi}_{i}\times \check
R^{ J}_{i})$ in the $ J$-direction, should be smaller than the quantity $\|\hat
R^{ J}\|$, representing the height of $\hat
R^{\phi}_{i}\times \hat R^{ J}_{i}$ in the $ J$-direction. By
construction $\|\check R^{\phi}_i\|=\check\gamma$,
$\|\check R^{ J}\|=\check\delta$ $\|\hat
R^{\phi}_i\|=\hat \gamma$,  and $\|\hat R^{
J}\|=\hat \delta$. By \eqref{eqn:const5} we have that
$K \check\delta>2\check\gamma+\hat\gamma+3\eps_1$, and by
\eqref{eqn:const7} we have $\hat\delta>\check\delta+\eps_1$.

We also need to ensure  the correct alignment of the
$d$-dimensional rectangle $\check R^{\theta}_{i}$, of exit
set $\partial\check R^{\theta }_{i}$,  with the
$d$-dimensional rectangle $\hat R^{\theta}_{i}$, of exit
set $\hat R^{\theta }_{i}$.   The correct alignment condition for windows that have the
exit sets consisting of their whole boundary means that the image of the first window contains the second window in its interior.
Thus, to make $\check R^{\theta }_{i}$ correctly
aligned with $\hat R^{\theta }_{i}$ under $F^{K}$, we
need to choose  $\hat R^{\theta }_{i}$ so that $\hat
R^{\theta }_{i}\subset \textrm{int}(\chi_\eps^{K }(\check
R^{\theta}_{i}))$.

By the product property  of correctly aligned windows, the outcome of this step is that   $\check W_{i}$ is correctly
aligned with $\hat W_i$ under   $F^{K}$.

\subsubsection{Step 4.} To construct a long sequence of correctly
aligned windows that visits all the sets
$\{\mathcal{L}^\eps_{i }\}_{i=0,\ldots, n_1}$, we start with
constructing a pair of  correctly aligned windows $W_0^-$ and
$W_1^+$, by the heteroclinic intersection $(\{z_{0,1}\}\times N)\cap\tilde\Gamma_\eps$ of
$W^u(\mathcal{L}^\eps_{0})$ and $W^s(\mathcal{L}^\eps_{1 })$, as in Step 1.  Then, as in Step 2, we construct the window $\hat W_0$
by $\mathcal{L}^\eps_{ J_0} $  and $\check W_1$ by $\mathcal{L}^\eps_{1} $. At this initial step, for the uniformity of the
notation, we set $\check W_0=\hat W_0$. Then, we continue the construction recursively, following Step 1, Step 2, and Step 3,  until we arrive with a
window $\check W_{n_1}$ by $\mathcal {L}_{n_1}$.
Hence, we obtain the  sequence of correctly aligned windows
\[ \hat W_0, W_0, W_1, \check W_1, \ldots, \hat W_{n_1-1}, W_{n_1-1}, W_{n_1}, \check{W}_{n_1}\]
starting from $\mathcal{L}^\eps_{0}$ and ending by $\mathcal{L}^\eps_{n_1}$. We apply the Shadowing Lemma type of result Theorem
\ref{theorem:detorb}, and conclude that there is an orbit
$\{y_i\}_{i=0,\ldots, n_1}$ with $y_i\in\check W_i$ for all $i$,  and
$y_{i+1}=F^{K +M +N }(y_i)$ for $i=0,\ldots,n_1-1$.

\subsection{Proof of Theorem \ref{thm:main1} and of Theorem \ref{thm:main2}}
For  Theorem \ref{thm:main1} the initial condition $\theta_0$ is chosen to be a non-trivial uniformly recurrent point.
For Theorem \ref{thm:main2} the initial condition $\theta_0$ can be any point in $N$. Then the construction from the previous section is performed.
The  outcome is an orbit $\{y_i\}_{i=0,\ldots,n_1}$ is $O(1)$ close to
the pseudo-orbit underlying the sequence of elementary building blocks  constructed in  Subsection \ref{subsection:gainseq}.
We can estimate the time it takes of an orbit $y_i$ to move from a
window $\check W_0$ to the window $\check W_{n_1}$.
The pseudo-orbit from Subsection \ref{subsection:gainseq} achieves a growth of scaled energy $\Delta H_\eps=O(\eps)$ in a scaled time $\Delta s=O(1/\eps^2)$. Each pseudo-orbit corresponding to a single elementary building block gives rise to a sequence of windows that are correctly aligned. The time  required by the correct alignment of windows, and the corresponding time it takes to the orbit to move along those windows, is $O(1)$, with the constants being the same for all windows in the sequence. Thus, the scaled time it takes the orbit $\{y_i\}_{i=0,\ldots,n_1}$ to travel from $y_0$ to $y_{n_1}$ is $O(1/\eps^2)$, equal to the time along the sequence of elementary building blocks multiplied by some constant independent of $\eps$, and hence on the energy.
The gain of physical energy along this orbit is $\Delta H=O(1/\eps)$, and the physical time that is spent is $\Delta t=O(1/\eps)$. Thus the change in the physical energy is
proportional to the time along the orbit, i.e. $\Delta t\approx
\Delta H\approx O(1/\eps)$, as claimed.

Since concatenations of correctly aligned windows are still correctly aligned, the above construction of correctly aligned windows can be continued for a time $O(1)$ to achieve a $O(1)$ change of the physical energy that covers the interval $H=\frac{1}{\eps^2}H_\eps\in[E_*,2E_*]$, where $E_*$ was the choice of initial energy.
To grow the physical energy to infinity, we   re-initialize the process staring with $\eps=1/\sqrt{2E_*}$, and we construct a sequence of correctly aligned windows as before. Again, the fact that concatenations of correctly aligned windows are still correctly aligned yields a longer sequence of correctly aligned windows. The construction of correctly aligned windows can be repeated inductively, yielding an infinite sequence of correctly aligned windows, and a corresponding shadowing orbit whose energy grows to infinity in time.   The estimate on the time does not get any worse since moving
up along the infinite chains corresponds to higher energy levels
where the speed of diffusion is only growing faster. Thus the energy growth is linear with respect to time.

\begin{rem} In  the above construction we should remark that the windows are of size $O(1)$, while the distance from
$\mathcal{L}^\eps_{i}$ to $\mathcal{L}^\eps_{i+1}$
is only of order $O(\eps^3)$. Thus, for an orbit
$\{y_i\}_{i=0,\ldots, n_1}$ that goes from $\check W_i$ about
$\mathcal{L}^\eps_{i}$ to $\check W_{i+1}$ about
$\mathcal{L}^\eps_{i+1}$ the net energy change is undetermined.
However, the method of correctly aligned windows does detect  an orbit whose physical energy changes by $O(1)$, from $E=E_*$ to $E=2E_*$, within a scaled time $O(1)$.    Hence, the topological argument is not capable to detect the
detailed changes of the energy along the orbit, but only the significant changes.

We could get more control on the energy levels visited by choosing the windows smaller but then we would get worse estimates on the  time.
This phenomenon resembles the `energy-time' uncertainty principle  of Heisenberg \cite{Heisenberg}.
\end{rem}

\subsection{Proof of Theorem \ref{thm:symbolic}}
The function $\mathcal{E}$ represents  a prescribed energy path with upper bounded derivative; this condition is necessary as the energy of the perturbed system cannot grow faster than linearly. The function $\mathcal{T}$ represents a parametrization of the time.  The argument is the same as for Theorem \ref{thm:main1}, provided that we choose the  infinite sequence of level sets  $\{J_\eps^{i}\}_{i\in\mathbb{N}}$ to follow the energy path $\mathcal{E}$, as described in Subsection \ref{subsection:reduction}.

\subsection{Regularity}\label{subsection:regularity}

We note that the if we assume that the metric and the potential
are $C^r$, we conclude that the flow is $C^{r-1}$ and so is the
map $F$. The theory of normal hyperbolicity requires at least $C^1$-differentiability.

Since we are using a derivative with respect to parameters and
even estimate the remainders, we would like that $r - 1 \ge 2$.

We note that the $C^1$-differentiability of the flow $\chi$ on the external manifold $N$
is sufficient for the argument, since the estimates do not involve any derivatives along the solution curves of $\chi$.

Hence, in Section \ref{section:main} it suffices to take $r_0 = 3$.

\section*{Acknowledgement} A part of this work has been done while M.G. and R.L. participated to the program Stability and Instability in Mechanical Systems, at the Centre de Recerca Matem\`atica, Barcelona, Spain. Another  part of this work has been done while M.G. was a member and R.L. was a visitor of the IAS. Both authors are very grateful to these two institutions. M.G. would also like to thank Keith Burns and Amie Wilkinson for useful discussions.

\end{document}